\documentclass[11pt]{article}
\usepackage[margin=1in]{geometry}
\usepackage[utf8]{inputenc}
\usepackage{authblk}
\usepackage{setspace}
\usepackage{graphicx}
\usepackage{subcaption}
\usepackage{amsmath}
\usepackage{amssymb}
\usepackage{amsthm}
\usepackage{bm}
\usepackage{algorithm}
\usepackage{algpseudocode}
\usepackage{tikz}
\usetikzlibrary{arrows}
\usepackage{booktabs}
\usepackage{multirow}
\usepackage{makecell}
\usepackage{diagbox}
\usepackage{enumitem} 
\usepackage[dvipsnames]{xcolor}

\usepackage[T1]{fontenc}
\usepackage[sc]{mathpazo}

\usepackage{natbib}
\bibpunct[, ]{(}{)}{;}{a}{,}{,}%

\newcommand{\E}{\mathbb{E}}

\newtheorem{assumption}{Assumption}
\newtheorem{remark}{Remark}
\newtheorem{lemma}{Lemma}
\newtheorem{theorem}{Theorem}
\newtheorem{corollary}{Corollary}

\newcommand{\keywords}[1]{%
  \vspace{0.5em}\noindent\textbf{Keywords:} \small #1\par}

\title{Feasibility Determination for Subjective Probability Constraints}

\usepackage{authblk}

\author[1]{Taehoon Kim}
\author[2]{Sigr\'un Andrad\'ottir}
\author[2]{Seong-Hee Kim}
\author[3]{Yuwei Zhou}

\affil[1]{Korea Military Academy, Republic of Korea Army}
\affil[2]{H. Milton Stewart School of Industrial and Systems Engineering, Georgia Institute of Technology}
\affil[3]{Kelley School of Business, Indiana University}

\date{}

\begin{document}

\maketitle

\vspace{-1cm}
\begin{abstract}
We consider the problem of determining feasible systems from a finite set of simulated alternatives with respect to probability constraints, where the observations from stochastic simulations are Bernoulli distributed. Most statistically valid procedures for feasibility determination focus on constraints on the means of normally distributed observations. Although these procedures can be adapted to Bernoulli-distributed data by treating batch means as basic observations, achieving approximate normality often requires a large batch size, potentially leading to the unnecessary waste of observations in reaching a decision. This paper proposes a procedure that utilizes the Bernoulli-distributed observations directly to determine feasibility. In addition, we incorporate subjective constraints, allowing for multiple thresholds for each constraint. We demonstrate that our proposed procedure is statistically valid and that it outperforms an existing feasibility determination procedure for subjective constraints originally developed for normally distributed observations. Furthermore, we propose two heuristic feasibility check approaches for thresholds that are sequentially added by decision makers, allowing thresholds to be tightened when many systems are feasible or relaxed when no feasible system exists. We show by experiments that the proposed procedures can efficiently provide feasibility decisions to systems with respect to all thresholds considered.

\end{abstract}


\section{Introduction}
\label{sec:Intro}

We consider the problem of identifying feasible systems among a finite number of simulated alternatives, when the observations follow Bernoulli distributions. This problem occurs when decision-makers consider constraints on probabilities. 
For instance, in a system operating under an $(s,S)$ inventory policy (where the decision maker orders products to increase the inventory level up to $S$ when the inventory level is below $s$ at a review period and places no order otherwise), a decision maker might wish to identify a combination of $s$ and $S$ that satisfies two constraints: (1) the probability of yearly total cost exceeding $1.4$ million dollars is no more than $h_1$, and (2) the probability that stockout happens within a year is no more than $h_2$. The decision maker may prefer strict thresholds $(h_1, h_2)=(1\%, 1\%)$. If no feasible solution exists, the decision maker may consider larger $h_1$ values such as 5\%, 10\%, or 20\%. Furthermore, if she is more sensitive to the stockout probability constraint, she may consider a tighter grid of $h_2$ values, such as 1\%, 2\%, 3\%, $\ldots, 20\%$, if necessary to identify a feasible system.         
In this example, the constraints are stochastic in the sense that the probability needs to be estimated based on stochastic observations. In addition, the basic observations are Bernoulli distributed with values 1 (the event of interest occurs) or 0 (the event of interest does not occur). The constraints are also subjective in that multiple threshold values can be considered. 

The feasibility determination problem we consider is studied in the field of ranking and selection (R\&S). R\&S procedures are primarily used to identify the system with the best performance measure among a finite number of simulated systems, where the definition of ``best'' depends on the specific problem; see \cite{Hong2021} for a review of R\&S. \cite{kim2006selecting} and \cite{Hong2015} discuss four types of selection problems: selection of the best, comparison with a standard, multinomial selection, and Bernoulli selection. Of these, the selection of the best, whose goal is to find the system with the largest or smallest expected performance measure, is studied the most. In the selection of the best, observations are typically assumed to be normally distributed, and several approaches have been developed. \cite{nelson2001simple} and \cite{KimNelson2001} consider the indifference zone (IZ) approach, where the IZ parameter represents the smallest difference worth detecting between the best system and the rest. \cite{chen2000simulation} and \cite{lee2010finding} propose optimal computing budget allocation (OCBA) procedures, while \cite{frazier2008knowledge} and \cite{XieFrazier2013} employ the Bayesian approach.  

When the decision maker aims to find a system with the largest probability of an event, Bernoulli selection can be applied, where observations follow Bernoulli distributions. A method for selecting the system with the highest probability of success using the IZ approach is introduced by \cite{sobel1957selecting}. Instead of using an IZ on the smallest difference worth detecting among probabilities, various researchers incorporate an IZ on the odds-ratio, which represents the number of successes per failure in each system. Specifically, \cite{bechhoferkiefersobel1968}, \cite{tamhane1985some}, \cite{paulson1994}, and \cite{wieland2004odds} develop Bernoulli selection procedures that utilize an odds-ratio IZ that represents the minimal odds-ratio worth detecting.
\cite{bechhoferkiefersobel1968} and \cite{tamhane1985some} solve the Bernoulli selection problem using the odds-ratio IZ with a random walk model. \cite{paulson1994} uses the odds-ratio IZ approach to eliminate inferior systems, and thus improve the performance of the selection procedure. Similarly, \cite{wieland2004odds} solve the Bernoulli selection problem using the gambler's ruin problem \citep{TaylorKarlin1994} combined with the odds-ratio IZ, aiming to develop a more efficient procedure by narrowing the decision-making region during the execution of the procedure. Alternatively, \cite{malone2005performance} apply existing procedures by \cite{KimNelson2001} for the selection of the best by taking batch means of Bernoulli data to achieve approximate normality and treating the batch means as basic observations.

While most R\&S procedures focus on a single performance measure, constrained R\&S considers multiple performance measures. The objective is to find the system with the best primary performance measure while satisfying constraints on the secondary performance measures. Thus, both feasibility determination and comparison are required in constrained R\&S. Several approaches have been developed to solve constrained R\&S. Among the procedures that utilize the IZ approach, \citet{AndradottirKim2010} and \citet{Healey2013, Healey2014} propose constrained R\&S procedures that find the best feasible system, while \citet{BaturKim2010} identify a set of feasible solutions in the presence of multiple constraints. \citet{hong2015chance} consider chance constrained selection of the best, where the secondary performance measures are probabilities. They propose procedures that first check the feasibility
of all solutions and then select the best among all the sample feasible solutions. \citet{Lee2012}, \citet{HunterPasupathy2013}, \citet{Pasupathy2014}, and \citet{GaoChen2017} propose sampling frameworks that approximate the optimal computing budget allocation while considering stochastic constraints. 
For the Bayesian approach, \citet{XieFrazier2013} discuss a Bayes-optimal policy for determining a set of simulated solutions whose mean performance exceeds a fixed threshold.

The aforementioned procedures consider a single fixed set of thresholds on the constraints. On the other hand, \cite{zhou2022finding} first introduce the concept of ``subjective" constraints, allowing multiple thresholds on each constraint to be considered at once. To reduce simulation costs, they recycle observations generated for feasibility determinations with different thresholds, which resembles the idea of ``green simulation" by \cite{feng2017green}. \cite{Zhou2024} extend the idea of subjective constraints to prune inferior systems as an intermediate step toward identifying the most preferred system in (possibly) multi-objective settings. Specifically, they allow the decision maker to impose multiple thresholds on each constraint, either by considering all thresholds simultaneously or by adding them sequentially. Inferior systems are then pruned based on feasibility decisions with respect to these thresholds, together with the decision maker’s preferences over different combinations of thresholds across constraints.

In constrained R\&S, constraints are typically imposed on the expectation of normally distributed data (except for \citet{hong2015chance}). Theoretically, when observations follow a Bernoulli distribution, one can still apply existing procedures for feasibility determination by treating the batch means of Bernoulli distributed data as basic observations. However, it is well known that large batch sizes can lead to inefficiency, causing unnecessary waste of observations before reaching a decision, especially in fully sequential type procedures, see \cite{kim2006asymptotic}.

In this paper, we develop IZ procedures for determining the feasibility of systems when constraints are on probabilities, i.e., the expectation of Bernoulli-distributed data with outputs of 1 (``success") or 0 (``failure"). We adopt an odds-ratio IZ, which is more practical than the conventional difference-based IZ (see Section \ref{subsec:CorrectDecision}). Our approach also eliminates the need to estimate variances or to batch observations to obtain approximate normality, and allows for subjective probability constraints, where multiple thresholds can be considered for each constraint. We prove that our approach guarantees statistical validity when all thresholds are defined up front and we demonstrate that it outperforms the existing feasibility determination procedure $\mathcal{RF}$ for subjective constraints with normally distributed observations due to \cite{zhou2022finding}.
Additionally, we propose heuristic multiple-pass procedures to allow decision makers to sequentially add thresholds if too many systems are feasible or if no feasible system is found. Our experiments show that the proposed procedures can be used for the efficient selection of the feasible system with respect to the most preferred combination of thresholds across multiple constraints.

The contributions of this paper are as follows: (i) we propose a procedure for feasibility determination with probability constraints using Bernoulli distributed data; (ii) we extend the analysis to incorporate subjective probability constraints; (iii) we prove the statistical validity of the proposed procedure and demonstrate empirically that it yields significant savings compared to an existing procedure; and (iv) we propose heuristic multiple-pass procedures that allow decision makers to add thresholds sequentially, and demonstrate their efficiency through experiments. 

It is worth mentioning that our problem is related to the best arm identification problem with fixed confidence in the multi-armed bandit (MAB) literature. Although R\&S and MAB share some similarities, MAB is different from our problem. MAB is typically considered in an online learning setting, with assumptions on bounded rewards or cost, and aims to minimize cumulative regrets, see \citet{Gabillon2011, Gabillon2012}, \citet{Kalyanakrishnan2012}, \citet{JamiesonNowak2014}, \citet{GarivierKaufmann2016}, \cite{Agrawal2021}, \citet{QinYou2025}, and their references for more detailed discussion. A closely related work to this paper is \citet{Locatelli2016} who propose an algorithm that finds the set of solutions whose means are above a given threshold under a fixed-budget setting (which is different from our setting). 
Some existing work also studies the constrained best arm identification problem; see \cite{Katz-Samuels2019}, \cite{Wang2022}, \cite{Li2023},  \cite{RussoVannella2024}, \cite{Dharod2024}, and \citet{Yang2025} for examples. To the best of our knowledge, none of the aforementioned literature consider subjective constraints. \citet{Auer2016} and \citet{Kone2023, Kone2024} consider the multi-objective best arm identification problem. However, they identify the Pareto set, whereas we identify feasible systems with respect to each combination of thresholds across all constraints. 

The rest of this paper is organized as follows: Section \ref{sec:background} provides our problem and notation. Our statistically valid procedure, along with three heuristic multi-pass procedures, are given in Section \ref{sec:procedure} and the proof of the statistical guarantee is included in Section \ref{sec:statistical validity}. Experimental results are discussed in Section \ref{sec:experiments}, followed by concluding remarks in Section \ref{sec:conclusion}. Appendices \ref{sec:AdditionalResults_BRFStoppingTime}, \ref{sec:RFComparison_Additional}, and \ref{sec:Experiments_Additional} include additional experimental results.
A detailed discussion of the tolerance levels used and the implementation of the competing procedure, $\mathcal{RF}$ by \citet{zhou2022finding}, is provided in Appendices \ref{sec:RF_Tolerance} and \ref{sec:RF}, respectively. Note that \citet{Kim2024} discuss a preliminary version of this work with only a single constraint and does not include either detailed mathematical proofs or the multi-pass procedures. \citet{Kim2024thesis} also discusses a preliminary version of this work.    

\section{Problem, Notation, and Correct Decision}
\label{sec:background}
In this section, we first discuss our problem and notation in Section \ref{subsec:Problem} and then define the correct decision event in Section \ref{subsec:CorrectDecision}. 

\subsection{Problem and Notation}
\label{subsec:Problem}

We consider stochastic and terminating simulations of $k$ systems, where each system is subject to $s$ constraints on probabilities. Let $\Gamma=\{1,\ldots,k\}$ denote the index set of all possible systems. Observation $Y_{i\ell n}$ represents whether an event of interest occurs ($Y_{i\ell n}=1$) or does not occur ($Y_{i\ell n} = 0$) during the $n$th replication of the $i$th system for the $\ell$th performance measure, where $i\in \Gamma$, $\ell = 1, 2, \ldots, s$, and $n = 1, 2, \ldots$. The probability of system $i$ regarding performance measure (constraint) $\ell$ is $p_{i\ell} = \E[Y_{i\ell n}]$, and system $i$ is considered feasible with respect to constraint $\ell$ at threshold $h$ if $p_{i\ell}\leq h$. Observations are assumed to satisfy the following assumption:
\begin{assumption} 
\label{assump:bern} 
For each $i \in \Gamma$, $(Y_{i1n}, Y_{i2n}, \ldots, Y_{isn})$ are independent and identically distributed for $n=1,2,\ldots$, where $Y_{i\ell n}$ are Bernoulli distributed with probability $p_{i\ell}$ for $\ell=1,\ldots,s$.
\end{assumption}

Note that we allow for cross-correlation between observations $Y_{i\ell n}$ and $Y_{i\ell' n}$ for $\ell \neq \ell'$ of different performance measures from the same replication of a system. For example, in an $(s, S)$ inventory example, the event that the yearly total cost exceeds a given threshold may be correlated with the event that a stockout occurs within a year. If common random numbers (CRN) are not used, then observation vectors from different systems, i.e., $(Y_{i1n}, Y_{i2n}, \ldots, Y_{isn})$ and $(Y_{i'1n}, Y_{i'2n}, \ldots, Y_{i'sn})$ for $i \neq i'$, are independent. 
Although CRN is generally not recommended for feasibility determination \citep{zhou2022finding}, we consider both independent sampling and CRN. This is because CRN is the default design in many commercial software packages, and feasibility determination procedures are often combined with a procedure for the selection of the best, where CRN is beneficial.

Our problem considers performing feasibility checks for subjective probability constraints, where the decision maker is willing to consider multiple threshold values for each constraint. When the decision maker is interested in considering a large number of possible thresholds or adjusting thresholds after reviewing the feasibility decisions for other thresholds, she may prefer to conduct feasibility checks across multiple passes with different threshold values. 
Specifically, for $w\geq 1$, let $d_\ell^{(w)}$ be the number of thresholds considered for constraint $\ell$ in the $w$th pass, and let $h_{\ell, m}^{(w)}$ be the $m$th threshold value being assessed, where $m=1,\ldots, d_\ell^{(w)}$. Without loss of generality, we assume $0< h_{\ell,1}^{(w)}< h_{\ell,2}^{(w)} < \cdots <h_{\ell,d_\ell^{(w)}}^{(w)}<1$ for $\ell=1,\ldots,s$. To conduct the multi-pass feasibility check, the decision maker performs feasibility check on thresholds $\{h_{\ell,1}^{(1)}, h_{\ell,2}^{(1)}, \ldots, h_{\ell,d_\ell^{(1)}}^{(1)}\}$ in a first pass. Based on the resulting feasibility decisions, she may perform a second pass on thresholds $\{h_{\ell,1}^{(2)}, h_{\ell,2}^{(2)}, \ldots, h_{\ell, d_\ell^{(2)}}^{(2)}\}$. This process can be repeated until the decision maker is satisfied with the feasibility decisions for the selected thresholds. 
To facilitate the discussion, for $w\geq 1$, we let 
\begin{align*}
    T_\ell^{(w)} &= \text{set of thresholds considered for constraint $\ell$ in the $w$th pass}= \{ h_{\ell, 1}^{(w)}, h_{\ell, 2}^{(w)}, \ldots, h_{\ell, d_\ell^{(w)}}^{(w)} \},  \\
    L^{(w)} &= \text{index set of constraints with thresholds to be considered in the $w$th pass (i.e., $\ell$ such } \\
    &\quad \text{that $T_\ell^{(w)}\not=\emptyset$).} 
\end{align*}
Our objective is to find the feasible systems with respect to each combination of thresholds across all constraints after completing $w$ feasibility check passes. 

\subsection{Correct Decision} 
\label{subsec:CorrectDecision}

We begin by discussing the IZ parameter. 
In Bernoulli selection, where the goal is to identify the system with the largest success probability among $k$ systems, three types of IZ settings are commonly considered \citep{wieland2004odds}. Since Bernoulli selection considers only one performance measure (unlike our feasibility problem), we now drop the subscript $\ell$ in $p_{i\ell}$, where $i\in \Gamma$. When $p_k > p_{k-1} \geq \cdots \geq p_1$, a statistically valid selection procedure guarantees to select system $k$ with at least $1-\alpha$ probability under one of the following three types of IZ settings: 

\begin{itemize}
    \item {\it Difference}: $p_k - p_{k-1} \geq \delta >0 $. 
    \item {\it Odds-ratio}: ${p_k/(1-p_k) \over p_{k-1}/(1-p_{k-1})} \geq \theta >1$.
    \item {\it Relative risk}: $p_k/p_{k-1} \geq \theta >1$.
\end{itemize}

As discussed in Section \ref{sec:Intro}, the odds-ratio is defined as the ratio of odds of 
successes to failure in one system relative to another system. We believe that the odds-ratio IZ $\theta>1$ is more practical than the difference IZ and relative risk IZ. For instance, with the difference IZ set to $\delta>0$, satisfying it would require $p_k\geq p_{k-1}+\delta$, which may require $p_k>1$. Similarly, to satisfy relative risk IZ with $\theta>1$, $p_k$ would need to be at least $\theta p_{k-1}$, which can also exceed 1. However, by using the odds-ratio IZ with $\theta>1$, the difference between $p_k$ and $p_{k-1}$ is significant if $p_k\geq \frac{\theta p_{k-1}}{1+(\theta-1)p_{k-1}}$. Since $\frac{\theta p_{k-1}}{1+(\theta-1)p_{k-1}}<1$ always holds, the requirement on $p_k$ is more practical. 
Furthermore, according to \cite{wieland2004odds}, the odds-ratio offers two key advantages. As the probability approaches 0 and 1, the odds-ratio reduces the difference considered significant among two probabilities. For instance, with $\theta=1.2$, if $p_{k-1} = 0.9$, then $p_k$ must be greater than $0.915$ for system $k$ to satisfy the odds-ratio IZ setting. Similarly, if $p_{k-1}=0.1$, then $p_k$ should be larger than 0.118. 
On the other hand, if $p_{k-1} = 0.5$, $p_k$ needs to exceed $0.545$. In other words, when measured by the odds-ratio, a 1\% difference between two probabilities close to 0 or 1 is considered a more significant difference than the same 1\% difference between two probabilities close to 50\%. This is desirable because increasing a probability close to 1 or decreasing a probability close to 0 requires more effort, and the odds-ratio accounts for this.
Finally, the odds-ratio enables the use of the gambler's ruin analysis approach, which does not require an initial sample size to estimate variance or a normality assumption of the observations (thus no batching is needed). 
For these reasons, we employ the odds-ratio IZ approach.

Unlike the Bernoulli selection problem, where the selection decisions are based on comparing observations across different systems, our feasibility check problem involves comparing observations from systems against user-specified thresholds (see the detailed discussion in Section \ref{sec:procedure}). Consequently, we adapt the odds-ratio IZ setting to the feasibility check context. 
Specifically, we consider an odds-ratio IZ parameter $\theta_\ell > 1$ for each constraint $\ell$ specified by the decision maker. During the $w$th pass, for a constraint where the probability $p_{i\ell}$ of system $i$ should be less than or equal to a threshold $h_{\ell, m}^{(w)}$, where $w\geq 1$ and $m=1,\ldots,d_\ell^{(w)}$, we define three sets:
\begin{itemize}
    \item Any system $i$ that satisfies 
    ${h_{\ell,m}^{(w)}/(1-h_{\ell,m}^{(w)}) \over p_{i\ell}/(1-p_{i\ell})}={(1-p_{i\ell})h_{\ell,m}^{(w)} \over p_{i\ell}(1-h_{\ell,m}^{(w)})} \ge \theta_\ell$
    is considered $D$esirable  with respect to threshold $h_{\ell,m}^{(w)}$ of constraint $\ell$, yielding the set $D_\ell(h_{\ell,m}^{(w)})$:
    \[
    D_\ell(h_{\ell,m}^{(w)})  =   \Bigg\{i\in \Gamma\; \Bigg|\;   {(1-p_{i\ell})h_{\ell,m}^{(w)} \over p_{i\ell}(1-h_{\ell,m}^{(w)})} \ge \theta_\ell \Bigg\}.
    \]
    \item  Any system $i$ with ${p_{i\ell}/(1-p_{i\ell}) \over h_{\ell,m}^{(w)}/(1-h_{\ell,m}^{(w)})}={p_{i\ell} (1-h_{\ell,m}^{(w)}) \over (1-p_{i\ell})h_{\ell,m}^{(w)}} \ge \theta_\ell$ is considered $U$nacceptable with respect to threshold $h_{\ell,m}^{(w)}$, placing them in set $U_\ell(h_{\ell,m}^{(w)})$:
    \[
    U_\ell(h_{\ell,m}^{(w)})  =   \Bigg\{i \in \Gamma\; \Bigg|\;  {p_{i\ell} (1-h_{\ell,m}^{(w)}) \over (1-p_{i\ell})h_{\ell,m}^{(w)}} \ge \theta_\ell \Bigg\}.
    \]
    \item The remaining systems are considered $A$cceptable and are placed in set $A_\ell(h_{\ell,m}^{(w)})$:
    \[
    A_\ell(h_{\ell,m}^{(w)})  =  \Gamma \setminus \left(D_\ell(h_{\ell,m}^{(w)})\cup U_\ell(h_{\ell,m}^{(w)})\right).
    \]
\end{itemize}
Figure \ref{fig:IZOddRatioDemo} shows a demonstration of the probabilities $p_{i\ell}$ that belong to the sets $D_\ell(h_{\ell,m}^{(w)}), U_\ell(h_{\ell,m}^{(w)})$, and $A_\ell(h_{\ell,m}^{(w)})$ for threshold $h_{\ell,m}^{(w)}$ on constraint $\ell$. We see that when $h_{\ell,m}^{(w)}$ is close to the center (i.e., around 0.5), the range of system means within $A_\ell(h_{\ell,m}^{(w)})$ is relatively wider compared to when $h_{\ell,m}^{(w)}$ is close to the extremes (i.e., around 0 or 1). This follows from the property of the odds-ratio that smaller differences are considered significant when $h_{\ell,m}^{(w)}$ is close to the extreme values.  
\begin{figure}[h!]
\centering
\includegraphics[scale=0.65]{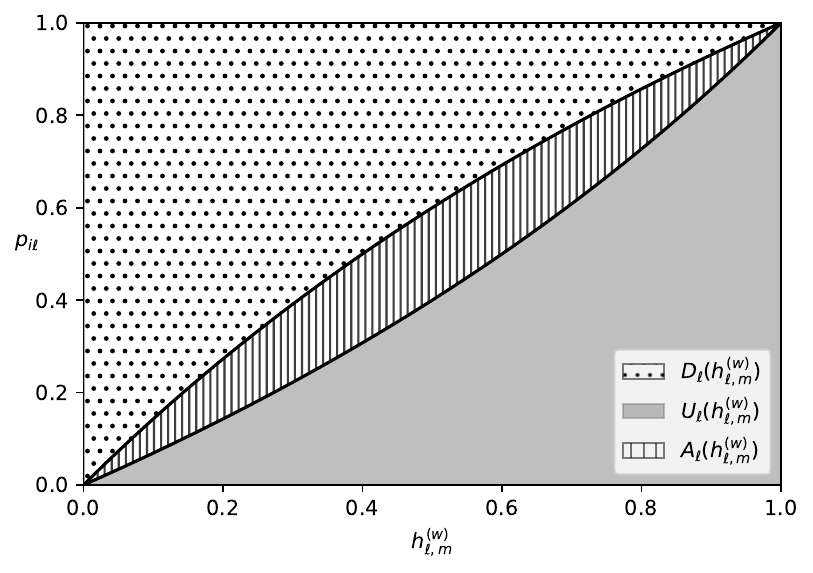}
\caption{Demonstration of the probabilities $p_{i\ell}$ that belong to the sets $D_\ell(h_{\ell,m}^{(w)}), U_\ell(h_{\ell,m}^{(w)})$, and $A_\ell(h_{\ell,m}^{(w)})$ as functions of the threshold $h_{\ell,m}^{(w)}$ on constraint $\ell$. $\theta_{\ell}$ is set as 1.5.} 
\label{fig:IZOddRatioDemo}
\end{figure}

When performing feasibility checks, we use ${\rm CD}_{i\ell}(h_{\ell,m}^{(w)})$ to denote the correct decision event for system $i$ with respect to threshold $h_{\ell,m}^{(w)}$ of constraint $\ell$. This is the event that system $i$ is declared feasible if $i \in D_\ell(h_{\ell,m}^{(w)})$ and infeasible if $i \in U_\ell(h_{\ell,m}^{(w)})$. For systems in $A_\ell(h_{\ell,m}^{(w)})$, any decision is considered correct. We define ${\rm CD}_{i\ell}$ as the correct decision event for system $i$ based on the thresholds tested in the first pass on constraint $\ell$, i.e., ${\rm CD}_{i\ell}=\cap_{h\in T_\ell^{(1)} } {\rm CD}_{i\ell}(h)$. Our proposed statistically-valid first-pass procedure ensures that the feasibility determination with respect to all thresholds tested in the first pass satisfies the following:
\begin{equation*}
\label{eqn:pcd}
\displaystyle{\rm PCD}=\Pr\left(\cap_{i=1}^k \cap_{\ell=1}^s {\rm CD}_{i\ell}\right) = \Pr\left(\cap_{i=1}^k \cap_{\ell=1}^s \cap_{h\in T_\ell^{(1)} } {\rm CD}_{i\ell}(h)\right)\geq 1-\alpha,
\end{equation*}
where $1 - \alpha$ is the nominal confidence level. 

In addition, we propose two heuristic procedures for subsequent passes when the decision maker decides to add thresholds. Although we do not prove the statistical validity of these procedures, we do not observe their validity being violated through experiments, see Section \ref{sec:experiments}. 

\begin{remark}
\label{remark:berf}
    Since the statistical validity of the first pass of our proposed procedure can be proved and its efficiency scales well with the number of thresholds considered (see Section \ref{sec:statistical validity} for more details), we recommend the decision maker to include all possible thresholds in the first pass (i.e., in $T_\ell^{(1)}$ for $\ell=1,\ldots,s$) if a statistically-valid procedure is desired.
\end{remark}

\section{Procedure} 
\label{sec:procedure}

In this section, we introduce our multi-pass Bernoulli feasibility check procedure $\mathcal{MPB}$. We let
$\mathcal{BRF}^{(w)}$ be the $w$th pass of Procedure $\mathcal{MPB}$, where $w=1,2,\ldots$. Section \ref{subsec:StatisticallyValidFirstPass} describes the statistically-valid first pass procedure $\mathcal{BRF}^{(1)}$. When the decision maker includes all possible thresholds within a single pass and does not conduct subsequent passes, we denote this procedure as $\mathcal{BRF}$.
Section \ref{subsec:HeuristicLaterPasses} discusses three heuristic procedures $\mathcal{BRF}_N^{(w)},\mathcal{BRF}_B^{(w)}$, and $\mathcal{BRF}_{BN}^{(w)}$ where $w\geq 2$, when decision maker chooses to add thresholds in subsequent passes.

\subsection{Statistically-Valid First-Pass Procedure}
\label{subsec:StatisticallyValidFirstPass}

Procedure $\mathcal{MPB}$ employs a random-walk model to process Bernoulli data without requiring the estimation of variance parameters or the use of batch means. Before presenting the procedure, we need some additional notation. First, for the overall confidence level $1-\alpha$, the nominal probability $\beta$ of error for each system is defined as follows:
\begin{equation}\label{eqn:beta}
    \beta = \left\{ \begin{array}{ll}
            1 - (1 - \alpha)^{1/k}, & \mbox{ if systems are simulated independently};\\
            \alpha/k, & \mbox{ if CRN is used}.
        \end{array} \right.
    \end{equation}
Next, let $\beta_\ell$ represent the nominal probability of error allocated to constraint $\ell$ within each system. The values of $\beta_\ell$ are determined as follows if $w$ feasibility check passes are performed: 
\begin{equation}
\label{eqn:betaell}
\begin{aligned}
& \text { $(i)$ }\; \beta_{\ell}=(\beta / s) \cdot \mathbb{I}\left(\sum_{u=1}^w d_{\ell}^{(u)}=1\right)+[\beta /(2 s)] \cdot \mathbb{I}\left(\sum_{u=1}^w d_{\ell}^{(u)}>1\right) \text { for } \ell=1 \text {, } \text { 2, } \ldots, s \text {, or } \\
& \text { $(ii)$ }\; \beta_{\ell}=\beta / D \text{ and } D=\sum_{\ell=1}^s \min \left\{\sum_{u=1}^w d_{\ell}^{(u)}, 2\right\} \text { for } \ell=1, \ldots, s, 
\end{aligned}
\end{equation}
where $\mathbb{I}(E)$ is the indicator function that returns 1 when event $E$ occurs and 0 otherwise.
Choice $(i)$ splits the nominal error $\beta$ equally among the constraints for each system and further among thresholds, while choice $(ii)$ splits $\beta$ equally among all ``effective'' thresholds tested across all constraints. As we prove in Section \ref{sec:statistical validity}, the statistical validity of Procedure $\mathcal{BRF}$ is based on the fact that there are at most two ``effective'' thresholds on each constraint $\ell$, regardless of the value of $d_\ell^{(1)}$. Note that the choices $(i)$ and $(ii)$ are identical if $\sum_{u=1}^w d_{\ell}^{(u)} >1$ for all $\ell=1,\ldots,s$.
\begin{remark}
    \label{remark:beta_all}
    In reality, the decision maker may not know the number of feasibility passes $w$ and $d_\ell^{(u)}$, for $u=2,\ldots, w$, in advance. In such cases, we recommend the decision maker to assume $\sum_{u=1}^w d_{\ell}^{(u)} > 1$ to determine $\beta_\ell$ if $d_\ell^{(1)}=1$ but there is a possibility for adding thresholds in the subsequent passes. 
\end{remark}

Since $\theta_\ell>1$, we can define $H_\ell$ as the smallest integer for constraint $\ell$ such that 
\begin{equation}
\beta_\ell \ge {1 \over 1 + \theta_\ell^{H_\ell}}.\label{eqn:H}
\end{equation}
We show that this choice of $H_\ell$ guarantees that the error for constraint $\ell$ does not exceed $\beta_\ell$ in Section~\ref{sec:statistical validity}.
Finally, for the $w$th pass, we let $I_{i\ell mn}^{(w)}$ denote dummy Bernoulli data with success probability $h_{\ell,m}^{(w)}$ for all $i\in\Gamma, \ell=1,\ldots,s, w\geq 1, m=1,\ldots,d_\ell^{(w)}$, and $n=1,2,\ldots$, which are independent of $Y_{i \ell 1},\ldots, Y_{i \ell n}$. 
Note that $I_{i\ell m n}^{(w)}$ can be correlated among different constraints (i.e., between $I_{i\ell m n}^{(w)}$ and $I_{i\ell' m n}^{(w)}$ where $\ell, \ell'=1,\ldots, s$), must be correlated among different thresholds on each constraint (i.e., between $I_{i\ell m n}^{(w)}$ and $I_{i\ell m' n}^{(w)}$ where $m, m'=1,\ldots,d_\ell^{(w)}$), and must be independent among different sampling stages (i.e., between $I_{i\ell mn}^{(w)}$ and $I_{i\ell m n'}$ where $n,n'=1,2,\ldots$). Moreover, if systems are simulated independently, $I_{i\ell m n}^{(w)}$ need to be independent among different systems (i.e., between $I_{i\ell m n}^{(w)}$ and $I_{i'\ell m n}^{(w)}$, where $i, i'=1,\ldots,k$). These assumptions are required for the proof of the statistical validity of $\mathcal{BRF}^{(1)}$.

Procedure $\mathcal{BRF}^{(w)}$ declares the feasibility of system $i$ with respect to $h_{\ell, m}^{(w)}$ for constraint $\ell$ as 
\begin{equation}
\begin{cases}
\text { feasible } & \text { if } \sum_{n=1}^{r_i}\left(Y_{i\ell n} -I_{i\ell mn}^{(w)}\right)\leq -H_\ell, \\ 
\text { infeasible } & \text { if } \sum_{n=1}^{r_i} \left(Y_{i\ell n} -I_{i\ell mn}^{(w)}\right)\geq H_\ell,  
\end{cases}  \label{eqn:BeRF_Decision1}
\end{equation}
where $r_i$ is the number of observations collected from system $i$ so far. 
In other words, the feasibility of system $i$ with respect to threshold $h_{\ell, m}^{(w)}$ for constraint $\ell$ is determined by where $\sum_{n=1}^{r_i}(Y_{i\ell n} -I_{i\ell mn}^{(w)})$ first exits the continuation region $(-H_\ell,H_\ell)$. Figure \ref{fig:berf} shows a sample path with $H_\ell=4$ where system $i$ is declared feasible with respect to threshold $h_{\ell, m}^{(w)}$ for constraint $\ell$.

\begin{figure}[h!]
\centering
\scalebox{0.25}{\includegraphics{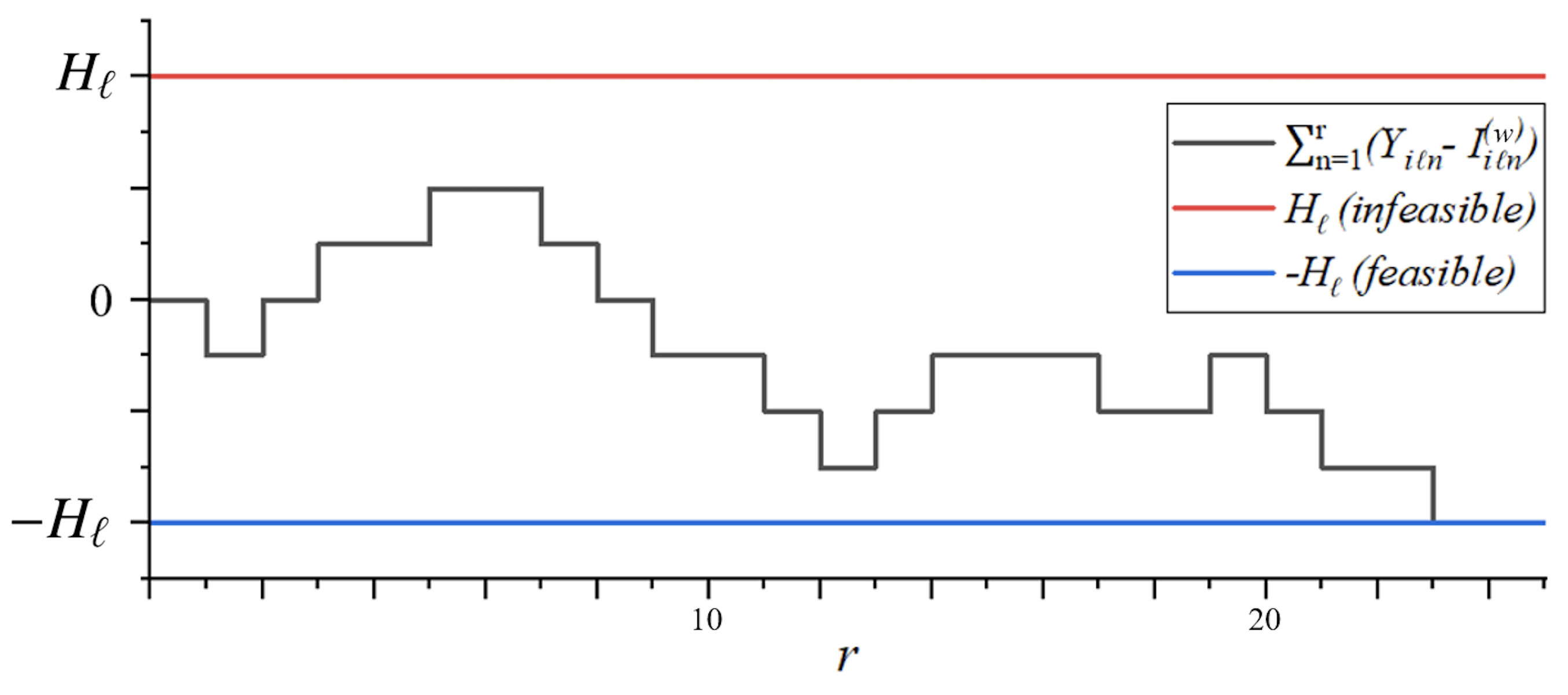}}
\caption{An example of a feasible decision with respect to the threshold $h_{\ell, m}^{(w)}$ for $\mathcal{BRF}^{(w)}$.} 
\label{fig:berf}
\end{figure}

\begin{remark}
   The dummy Bernoulli variables $I_{i\ell mn}^{(w)}$ introduce additional randomness (beyond the inherent randomness in the simulation data $Y_{i\ell n}$). At first glance, this may seem undesirable. However, the dummy Bernoulli variables $I_{i \ell mn}^{(w)}$ allow us to apply a simple random walk model, eliminating the need for initial observations in variance estimation. 
   Alternatively, one could directly use the threshold constant $h_{\ell,m}^{(w)}$ for feasibility determination. However, this approach leads to an infinite state space, making the application of a simple random walk model infeasible. Instead, the monitoring statistics can be approximated using Brownian motions. As demonstrated in Section \ref{subsec:RFComparison}, where we compare the performance of $\mathcal{BRF}$ with $\mathcal{RF}$ (which employs a Brownian motion approximation), our results indicate that incorporating the dummy variables $I_{i \ell mn}^{(w)}$ generally yields better performance. 
\end{remark}

Note that Equation (\ref{eqn:BeRF_Decision1}) is equivalent to declaring system $i$ 
\begin{equation}
\begin{cases}\text { feasible } & \text { if } \bar{Y}_{i\ell}(r_i)+\frac{H_\ell}{r_i} \leq \frac{\sum_{n=1}^{r_i} I_{i\ell mn}^{(w)}}{r_i}  , \\ 
\text { infeasible } & \text { if } \bar{Y}_{i\ell}(r_i)-\frac{H_\ell}{r_i} \geq \frac{\sum_{n=1}^{r_i} I_{i\ell mn}^{(w)}}{r_i},
\end{cases} \label{eqn:BeRF_Decision2}
\end{equation}
where $\bar{Y}_{i\ell}(r_i)$ denotes the average value of $r_i$ observations from system $i$ for constraint $\ell$, i.e., $\bar{Y}_{i\ell}(r_i)=\sum_{n=1}^{r_i} Y_{i\ell n}/r_i$. Equation (\ref{eqn:BeRF_Decision2}) provides an alternative interpretation of feasibility determination, which is used for performing feasibility checks for added thresholds in later passes (see Section \ref{subsec:HeuristicLaterPasses}). Specifically, the interval $[\bar{Y}_{i\ell}(r_i)+H_\ell/r_i,\bar{Y}_{i\ell}(r_i)-H_\ell/r_i]$ is the same for all threshold values and is updated as observations are collected. System $i$ is declared feasible (infeasible) for constraint $\ell$ with respect to threshold $h_{\ell,m}^{(w)}$ as soon as $\sum_{n=1}^{r_i} I_{i\ell mn}^{(w)}/r_i$ is excluded from the interval $[\bar{Y}_{i\ell}(r_i)-H_\ell/r_i, \bar{Y}_{i\ell}(r_i)+H_\ell/r_i]$ through its upper (lower) boundary. 
Figure \ref{fig:BeRF_Boundary} shows an example where two thresholds $h_{\ell, 1}^{(1)}<h_{\ell, 2}^{(1)}$ are considered during the first pass. System $i$ is declared infeasible with respect to threshold $h_{\ell,1}^{(1)}$ after $r_{i\ell 1}^{(1)}$ observations, as this is the first time $\sum_{n=1}^{r_i} I_{i\ell 1 n}^{(1)}/r_i$ is excluded from the interval $[\bar{Y}_{i\ell}(r_i)+H_\ell/r_i,\bar{Y}_{i\ell}(r_i)-H_\ell/r_i]$ through its lower boundary. Similarly, system $i$ is declared feasible with respect to threshold $h_{\ell, 2}^{(1)}$ after $r_{i\ell 2}^{(1)}$ observations, as $\sum_{n=1}^{r_i} I_{i\ell 2 n}^{(1)}/r_i$ is excluded from the interval through its upper boundary for the first time.

\begin{figure}[h!]
    \centering
    \includegraphics[scale=0.65]{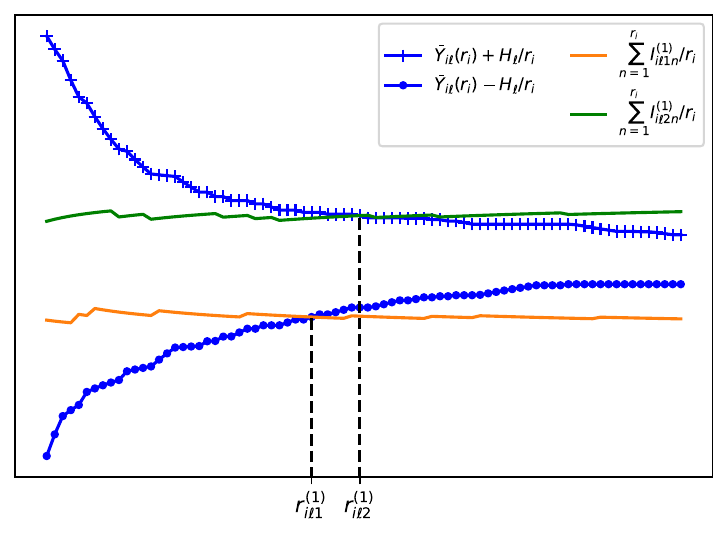}
    \caption{Feasibility determination for thresholds $h_{\ell, 1}^{(1)}<h_{\ell, 2}^{(1)}$ on constraint $\ell$.}
    \label{fig:BeRF_Boundary}
\end{figure}

In practice, the decision maker may wish to add new thresholds for constraint $\ell$ after the first pass of feasibility check. For example, as discussed in Section \ref{sec:Intro}, a decision maker considers a constraint on the probability of a stockout occurring within a year under an $(s, S)$ inventory policy. Suppose she starts with thresholds $\{0.05, 0.1, 0.2\}$ in the first pass and identifies multiple $s$ and $S$ combinations whose stockout probability lies between 0.05 and 0.1. Then she can add thresholds $\{0.06, 0.07, 0.08, 0.09\}$ and perform a second-pass feasibility check to further narrow down the stockout probability of the candidate systems.
To recycle the observations collected from earlier passes of feasibility checks, we may keep all values of $\bar{Y}_{i\ell}(r)$ for $r=1,2,\ldots,r_i$ until all passes are completed. However, this is undesired due to data storage problems. Inspired by \citet{Zhou2024}, we instead keep two statistics while the simulation of system $i$ is ongoing: 
\begin{equation*}
    v_{i\ell}^{\rm UB} = \min \left\{\bar{Y}_{i\ell}(r_i)+\frac{H_\ell}{r_i} \right\} \quad \text{ and } \quad
    v_{i\ell}^{\rm LB} = \max \left\{\bar{Y}_{i\ell}(r_i)-\frac{H_\ell}{r_i} \right\}.
\end{equation*}
We also introduce variable ${\rm LAST}_{i\ell}$, which tracks whether $v_{i\ell}^{\rm UB}$ or $v_{i\ell}^{\rm LB}$ changes when a new observation is collected. 
In Section \ref{subsec:HeuristicLaterPasses}, we explain how we use $v_{i\ell}^{\rm LB},v_{i\ell}^{\rm UB}$, and ${\rm LAST}_{i\ell}$ to conclude feasibility decisions for the added thresholds in later passes. 
While the variables $v_{i\ell}^{\rm UB},v_{i\ell}^{\rm LB}$, and ${\rm LAST}_{i\ell}$ are recorded in the first pass, they do not influence feasibility decisions in the first pass.

Algorithm \ref{alg:berf} describes the proposed first-pass procedure, denoted as $\mathcal{BRF}^{(1)}$ (we express the steps in terms of $w\geq 1$ to facilitate the description of our procedures for $w\geq 2$ in Section \ref{subsec:HeuristicLaterPasses}). In the algorithm, we use $Z_{i\ell m}^{(w)}$ to indicate whether system $i$ is feasible ($Z_{i\ell m}^{(w)}=1$) or infeasible ($Z_{i\ell m}^{(w)}=0$) with respect to threshold $h_{\ell, m}^{(w)}$ for constraint $\ell$, where $ i\in \Gamma, \ell\in L^{(1)}, w=1$, and $m=1, \ldots, d_\ell^{(w)}$. 
We use a common random seed to generate $I_{i\ell mn}^{(w)}$ for each system $i$ and observation $n$, which implies that $I_{i\ell mn}^{(w)}$ are dependent across constraints (i.e., $\ell$) and thresholds (i.e., $m$).
\begin{algorithm}[h!]
	\caption{First Pass of Procedure Bernoulli Recycled Feasibility, $\mathcal{BRF}^{(w)}, w=1$}\label{alg:berf}
{\fontsize{10}{13}\selectfont	
\begin{algorithmic}
	\State [{\bf Setup}:] Choose confidence level $1-\alpha$, threshold set $T_\ell^{(w)}=\{h_{\ell, 1}^{(w)},h_{\ell, 2}^{(w)},\ldots,h_{\ell, d_\ell^{(w)}}^{(w)}\}$, and odds-ratio IZ parameters $\theta_\ell$ for $\ell=1,2,\ldots, s$. Set $\Gamma =\{1,2,\ldots,k\}$, $L^{(w)}=\{\ell=1,\ldots,s\mid T_\ell^{(w)}\not=\emptyset\}$, and $H_\ell$ as the smallest integer such that \eqref{eqn:H} holds,
    where $\beta_\ell$ is determined as in \eqref{eqn:betaell} (see also Remark \ref{remark:beta_all}).

	\For{each system $i \in \Gamma$} 
        \State [{\bf Initialization}:] 
        \begin{itemize}
            \item Set ${\rm ON} = L^{(w)}$ and ${\rm ON}_\ell = \{1, 2, \ldots, d_\ell^{(w)}\}$ for $\ell\in {\rm ON}$.
            \item Assign system $i$ random seeds ${\rm SEED}^y_i$ and ${\rm SEED}^u_i$.
            \item Set $r_i=1$, generate $Y_{i\ell r_i}$ using ${\rm SEED}^y_i$ for $\ell=1,2,\ldots,s$, and generate a uniform random variable $U_{i r_i}\sim U(0,1)$ using ${\rm SEED}^u_i$.  
            \item Set $v_{i\ell}^{\rm UB}=\infty,v_{i\ell}^{\rm LB}=-\infty$, and ${\rm LAST}_{i\ell}$ as an empty string for $\ell=1,\ldots,s$.
        \end{itemize}
        \State [{\bf Feasibility Check}:]  
        \For{each constraint $\ell \in {\rm ON}$}
            \begin{itemize}
            \setlength{\itemindent}{0.15in}
                \item[] Set $v_{i\ell}^{\rm LB}=\max \left\{v_{i\ell}^{\rm LB}, \bar{Y}_{i\ell}(r_i)-H_\ell/r_i\right\}$. If $v_{i\ell}^{\rm LB}$ is updated, set ${\rm LAST}_{i\ell}={\rm LB}$.
                \item[] Set $v_{i\ell}^{\rm UB}=\min \left\{v_{i\ell}^{\rm UB}, \bar{Y}_{i\ell}(r_i)+H_\ell/r_i\right\}$. If $v_{i\ell}^{\rm UB}$ is updated, set ${\rm LAST}_{i\ell}={\rm UB}$.
            \end{itemize}
        \For{each threshold $m \in {\rm ON}_\ell$}
		\begin{itemize}
			\setlength{\itemindent}{0.4in}
            \item[] Set $I_{i \ell m r_i}^{(w)}=1$ if $U_{i r_i} \le h_{\ell, m}^{(w)}$ and 0 otherwise.
            \item[] If $\bar{Y}_{i\ell}(r_i)+H_\ell/r_i \leq \sum_{n=1}^{r_i} I_{i \ell m n}^{(w)}/r_i$, set $Z_{i\ell m}^{(w)}=1$ and ${\rm ON}_\ell = {\rm ON}_\ell\setminus \{m\}$;
			\item[] Else if $\bar{Y}_{i\ell}(r_i)-H_\ell/r_i\geq \sum_{n=1}^{r_i} I_{i \ell m n}^{(w)}/r_i$, set $ Z_{i\ell m}^{(w)}=0$ and ${\rm ON}_\ell = {\rm ON}_\ell\setminus \{m\}$. 
		\end{itemize}
        \EndFor     
        \State If ${\rm ON}_\ell = \emptyset$, set ${\rm ON} = {\rm ON}\setminus \{\ell\}$.
        \EndFor
    \State [{\bf Stopping Condition}:] 
		 \begin{itemize}
			\item[] If ${\rm ON}=\emptyset$, return $Z_{i\ell m}^{(w)}$ for $\ell\in L^{(w)}$ and $m=1,\ldots,d_\ell^{(w)}$ and the number of observations collected $r_i$, and save the assigned (initial) random seeds ${\rm SEED}^y_i$ and ${\rm SEED}^u_i$ for system $i$. Otherwise, set $r_i=r_i+1$, obtain $Y_{i\ell r_i}$ for all $\ell=1,\ldots,s$, generate $U_{ir_i}\sim U(0,1)$ independent of $Y_{i\ell 1},\ldots, Y_{i\ell r_i}$, and go to [{\bf Feasibility Check}]. 
		\end{itemize} 
    \EndFor
    \end{algorithmic}
}
\end{algorithm}


\begin{remark}
    If the decision maker only wishes to perform feasibility check through one pass, i.e., Procedure $\mathcal{BRF}$, there is no need to keep track of $v_{i\ell}^{\rm UB}, v_{i\ell}^{\rm LB}$, and ${\rm LAST}_{i\ell}$ when implementing Algorithm \ref{alg:berf}. Likewise, there is no need to store the random seeds used to generate the observations or the dummy Bernoulli variables; these are only needed when additional thresholds are introduced in subsequent passes. For brevity, we omit the algorithm statement for Procedure $\mathcal{BRF}$.
\end{remark}

\begin{remark}
    In Algorithm \ref{alg:berf}, we assign unique random seeds ${\rm SEED}^y_i$ and ${\rm SEED}^u_i$ to system $i$ to avoid the dependency among systems. When the systems are simulated under CRN, one can simply use common random seeds for all systems, i.e., by setting ${\rm SEED}_i^y={\rm SEED}^y$ and ${\rm SEED}_i^u={\rm SEED}^u$ for all $i\in\Gamma$.
\end{remark}

\subsection{Heuristic Procedures for Later Passes}
\label{subsec:HeuristicLaterPasses}

In this section, we discuss three heuristic approaches for adding thresholds for feasibility determination in later passes. The first two approaches are referred as $\mathcal{BRF}_B^{(w)}$ ($B$ for Bernoulli) and $\mathcal{BRF}_N^{(w)}$ ($N$ for normal), where $w\geq 2$, and we will discuss each one separately. In addition, we also propose a heuristic approach $\mathcal{BRF}_{BN}^{(w)}$ that combines  $\mathcal{BRF}_B^{(w)}$ and $\mathcal{BRF}_N^{(w)}$.

\paragraph{Heuristic approach $\mathcal{BRF}_B^{(w)}$.}
$\mathcal{BRF}_B^{(w)}$ employs similar decision criteria as Algorithm \ref{alg:berf} with a key difference: the feasibility decisions are based on comparing $\sum_{n=1}^{r_i} I_{i\ell mn}^{(w)}/r_i$ with $v_{i\ell}^{\rm LB}$ and $v_{i\ell}^{\rm UB}$ instead of comparing with $\bar{Y}_{i\ell}(r_i)+H_\ell/r_i$ and $\bar{Y}_{i\ell}(r_i)-H_\ell/r_i$. 
Before performing the feasibility check, we generate $I_{i\ell mn}^{(w)}$, for $n=1,\ldots,r_i$, using the same (initial) random seed assigned to system $i$ as in the previous pass $\mathcal{BRF}^{(w-1)}$ (i.e., ${\rm SEED}^u_i$). This ensures that $\mathcal{BRF}_B^{(w)}$ utilizes the same values of $I_{i\ell mn}^{(w)}$ (for $1\leq n\leq r_i$) as if $h_{\ell,m}^{(w)}$ had been included in the previous pass.
There are four possible outcomes when $\mathcal{BRF}_B^{(w)}$ is initialized: 
\begin{itemize}
    \item {\bf When $v_{i\ell}^{\rm UB}\leq \sum_{n=1}^{r_i} I_{i\ell mn}^{(w)}/r_i$ and $v_{i\ell}^{\rm LB}<\sum_{n=1}^{r_i} I_{i\ell mn}^{(w)}/r_i$:} The system $i$ is immediately declared feasible with respect to threshold $h_{\ell,m}^{(w)}$.
    \item {\bf When $v_{i\ell}^{\rm LB} \geq \sum_{n=1}^{r_i} I_{i\ell mn}^{(w)}/r_i$ and $v_{i\ell}^{\rm UB}>\sum_{n=1}^{r_i} I_{i\ell mn}^{(w)}/r_i$:} The system $i$ is immediately declared infeasible with respect to threshold $h_{\ell,m}^{(w)}$.
    \item {\bf When $v_{i\ell}^{\rm UB}\leq \sum_{n=1}^{r_i} I_{i\ell mn}^{(w)}/r_i\leq v_{i\ell}^{\rm LB}$:} The system $i$ is declared feasible (infeasible) with respect to threshold $h_{\ell,m}^{(w)}$ based on the value of ${\rm LAST}_{i\ell}$. If ${\rm LAST}_{i\ell}={\rm LB}$, system $i$ is deemed feasible; if ${\rm LAST}_{i\ell}={\rm UB}$, system $i$ is deemed infeasible. 

    Note that, although $v_{i\ell}^{\rm LB}<v_{i\ell}^{\rm UB}$ happens in general, it is possible for $v_{i\ell}^{\rm UB}\leq v_{i\ell}^{\rm LB}$ when $\mathcal{BRF}^{(w-1)}$ has concluded feasibility decisions for all thresholds considered on constraint $\ell$. In such cases, the decision depends on ${\rm LAST}_{i\ell}$. Specifically, ${\rm LAST}_{i\ell}={\rm LB}$ implies that $v_{i\ell}^{\rm LB}$ updates when $\mathcal{BRF}^{(w-1)}$ terminates, meaning that $v_{i\ell}^{\rm UB}\leq h_{\ell,m}^{(w)}$ would have been satisfied if $h_{\ell,m}^{(w)}$ has been considered in the $(w-1)$th pass. Therefore, we directly declare system $i$ feasible with respect to threshold $h_{\ell,m}^{(w)}$ if ${\rm LAST}_{i\ell}={\rm LB}$.
    Similarly, if ${\rm LAST}_{i\ell}={\rm UB}$, system $i$ is declared infeasible with respect to threshold $h_{\ell,m}^{(w)}$.  
    \item {\bf When $v_{i\ell}^{\rm LB}< \sum_{n=1}^{r_i} I_{i\ell mn}^{(w)}/r_i < v_{i\ell}^{\rm UB}$:} We collect additional observations, generate $I_{i\ell m r_i'}^{(w)}$ for $r_i'=r_i+1,r_i+2,\ldots$, and update variables $v_{i\ell}^{\rm UB}, v_{i\ell}^{\rm LB}$, and ${\rm LAST}_{i\ell}$ accordingly until the feasibility decision is made when either $v_{i\ell}^{\rm UB}\leq \sum_{n=1}^{r_i} I_{i\ell mn}^{(w)}/r_i$ or $v_{i\ell}^{\rm LB}\geq \sum_{n=1}^{r_i} I_{i\ell mn}^{(w)}/r_i$ holds. 
\end{itemize}

Figure \ref{fig:BeRF_Heuristic2} shows an example of how $\mathcal{BRF}_B^{(2)}$ concludes feasibility decisions for two added thresholds during the second pass, while the first pass is as in Figure \ref{fig:BeRF_Boundary} and $T_\ell^{(2)}=\{h_{\ell, 1}^{(2)}, h_{\ell, 2}^{(2)}\}$. 
System $i$ is immediately concluded to be infeasible with respect to $h_{\ell,1}^{(2)}$ since $\sum_{n=1}^{r_i} I_{i\ell 1n}^{(2)}/r_i<v_{i\ell}^{\rm LB}$. We let ${r^{(2)}_{i\ell 1}}={r^{(1)}_{i\ell 2}}$ be the required number of observations to conclude feasibility decision for $h_{\ell,1}^{(2)}$.
For threshold $h_{\ell, 2}^{(2)}$, additional observations are collected until either $v_{i\ell}^{\rm UB}\leq \sum_{n=1}^{r_i} I_{i\ell 2n}^{(2)}/r_i$ or $v_{i\ell}^{\rm LB}\geq \sum_{n=1}^{r_i} I_{i\ell 2 n}^{(2)}/r_i$ is satisfied. In this example, this condition is not satisfied until $r^{(2)}_{i\ell 2}$ observations are obtained, after which system $i$ is declared feasible with respect to $h_{\ell,2}^{(2)}$. Algorithm \ref{alg:berf_b} presents the statement of the $w$th pass procedure $\mathcal{BRF}_B^{(w)}$, where $w\geq 2$.

\begin{figure}[h!]
\centering
\begin{subfigure}{.5\textwidth}
  \centering
  \includegraphics[width=0.9\linewidth]{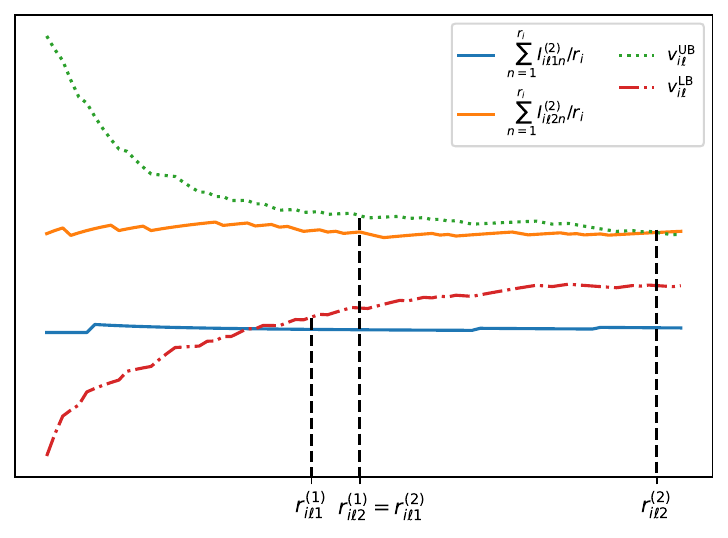}
  \caption{$\mathcal{BRF}_B^{(2)}$}
  \label{fig:BeRF_Heuristic2}
\end{subfigure}%
\begin{subfigure}{.5\textwidth}
  \centering
  \includegraphics[width=0.9\linewidth]{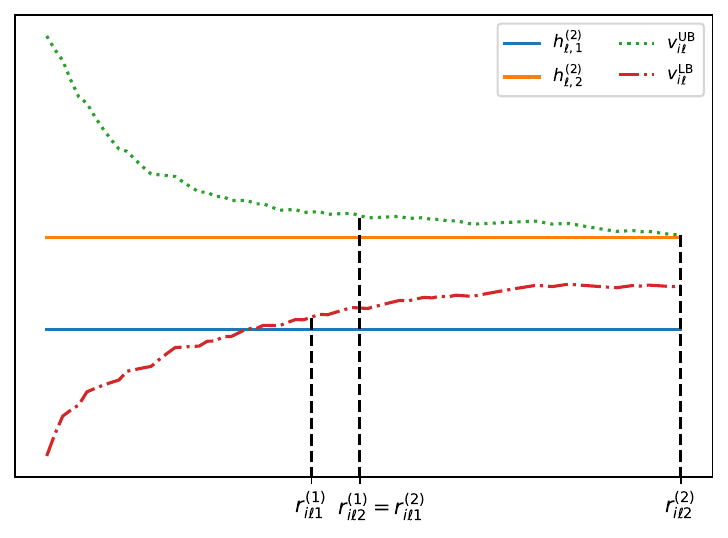}
  \caption{$\mathcal{BRF}_N^{(2)}$}
  \label{fig:BeRF_Heuristic1}
\end{subfigure}
\caption{An example of how $\mathcal{BRF}_B^{(2)}$ and $\mathcal{BRF}_N^{(2)}$ conclude feasibility decisions for two added thresholds $h_{\ell,1}^{(2)}<h_{\ell,2}^{(2)}$ during the second pass.}
\label{fig:BeRF_Heuristic}
\end{figure}

\begin{algorithm}[h!]
{\fontsize{10}{13}\selectfont	
    \begin{algorithmic}
    \caption{The $w$th Pass of Procedure Bernoulli Recycled Feasibility, $\mathcal{BRF}_B^{(w)}$, where $w\geq 2$}
    \label{alg:berf_b}
		\State {\bf [Setup:]} Choose the threshold set $T_\ell^{(w)}=\{h_{\ell,1}^{(w)},\ldots,h_{\ell,d_\ell^{(w)}}^{(w)}\}$ on constraint $\ell=1,\ldots,s$ for the $w$th pass. Set $L^{(w)}=\{\ell=1,\ldots,s\mid T_\ell^{(w)}\not=\emptyset\}$. 
		\For{each system $i\in\Gamma$}
		\State [{\bf Initialization}:] 
		\begin{itemize}
            \item Set ${\rm ON}=L^{(w)}$ and ${\rm ON}_\ell=\{1,\ldots,d_\ell^{(w)}\}$ for $\ell\in {\rm ON}$. 
            \item Obtain $r_i, \bar{Y}_{i\ell}(r_i), {\rm LAST}_{i\ell}, v_{i\ell}^{\rm LB}, v_{i\ell}^{\rm UB}, {\rm SEED}_i^y$, and ${\rm SEED}_i^u$ for $i\in\Gamma$ and $\ell=1,\ldots,s$ from $\mathcal{BRF}^{(1)}$ (if $w=2$) or  $\mathcal{BRF}_B^{(w-1)}$ (if $w\geq 3$). 
            \item For $\ell\in L^{(w)}$ and $m\in {\rm ON}_\ell$, generate $I_{i\ell m n}^{(w)}$ for $n=1,\ldots,r_i$ using ${\rm SEED}^u_i$ and compute $\sum_{n=1}^{r_i} I_{i\ell m n}^{(w)}/r_i$.
		\end{itemize}
        \State [{\bf Initial Feasibility Check} for $\mathcal{BRF}_B^{(w)}$:]
        \For{each constraint $\ell\in{\rm ON}$}   	
        \For{each threshold $m\in {\rm ON}_\ell$}
        \State If $v_{i\ell}^{\rm UB}\leq \sum_{n=1}^{r_i} I_{i\ell m n}^{(w)}/r_i$ and $v_{i\ell}^{\rm LB}<\sum_{n=1}^{r_i} I_{i\ell m n}^{(w)}/r_i$, set $Z_{i\ell m}^{(w)}=1$ and ${\rm ON}_\ell={\rm ON}_\ell \setminus \{m\}$;
        \State Else if $v_{i\ell}^{\rm LB}\geq \sum_{n=1}^{r_i} I_{i\ell m n}^{(w)}/r_i$ and $v_{i\ell}^{\rm UB}>\sum_{n=1}^{r_i} I_{i\ell m n}^{(w)}/r_i$, set $Z_{i\ell m}^{(w)}=0$ and ${\rm ON}_\ell={\rm ON}_\ell \setminus \{m\}$;
        \State Else if $v_{i\ell}^{\rm UB}\leq \sum_{n=1}^{r_i} I_{i\ell m n}^{(w)}/r_i\leq v_{i\ell}^{\rm LB}$,
            \begin{itemize}[leftmargin=2.5cm]
                \item[] if ${\rm LAST}_{i\ell}={\rm UB}$, set $Z_{i\ell m}^{(w)}=0$ and ${\rm ON}_\ell={\rm ON}_\ell \setminus \{m\}$;
                \item[] if ${\rm LAST}_{i\ell}={\rm LB}$, set $Z_{i\ell m}^{(w)}=1$ and ${\rm ON}_\ell={\rm ON}_\ell \setminus \{m\}$.
            \end{itemize}
        \EndFor
        \State If ${\rm ON}_\ell=\emptyset$, set ${\rm ON}={\rm ON} \setminus \{\ell\}$.
    \EndFor
    \State [{\bf Stopping Condition} for $\mathcal{BRF}_B^{(w)}$:] If ${\rm ON}=\emptyset$, return $Z_{i\ell m}^{(w)}$ for $\ell\in L^{(w)}$ and $m=1,\ldots,d_\ell^{(w)}$, the number of observations collected $r_i$, and the (initial) random seeds ${\rm SEED}^y_i$ and ${\rm SEED}^u_i$. Otherwise, set $r_i=r_i+1$, take one additional observation $Y_{i\ell r_i}$ for all $\ell\in {\rm ON}$, use ${\rm SEED}^u_i$ to generate $U_{i r_i}\sim U(0,1)$. 
    Then go to [{\bf Feasibility Check}] in Algorithm \ref{alg:berf}.

    
	\EndFor%
	\end{algorithmic}
 }
\end{algorithm}

\vspace{-0.5cm}
\paragraph{Heuristic approach $\mathcal{BRF}_N^{(w)}$.}
$\mathcal{BRF}_N^{(w)}$ utilizes the fact that, with a sufficiently large number of Bernoulli observations collected by the end of the first pass, the sample means are approximately normally distributed. In this case, we perform feasibility checks for the newly-added thresholds similar to the multi-pass procedure from \citet{Zhou2024}, which is designed for normally distributed observations. Instead of comparing $v_{i\ell}^{\rm LB}$ and $v_{i\ell}^{\rm UB}$ with $\sum_{n=1}^{r_i} I_{i\ell mn}^{(w)}/r_i$, as in Algorithm \ref{alg:berf}, we compare these values with the newly-added thresholds $h_{\ell, m}^{(w)}$, where $m=1,\ldots,d_\ell^{(w)}$, during the $w$th pass. There are four possible outcomes: (1) when $v_{i\ell}^{\rm UB}\leq h_{\ell,m}^{(w)}$ and $v_{i\ell}^{\rm LB}<h_{\ell,m}^{(w)}$; (2) when $v_{i\ell}^{\rm LB} \geq h_{\ell,m}^{(w)}$ and $v_{i\ell}^{\rm UB}>h_{\ell,m}^{(w)}$; (3) when $v_{i\ell}^{\rm UB}\leq h_{\ell,m}^{(w)}\leq v_{i\ell}^{\rm LB}$; and (4) when $v_{i\ell}^{\rm LB}< h_\ell^{(m)} < v_{i\ell}^{\rm UB}$. The decision criteria match those in $\mathcal{BRF}_B^{(w)}$  except for replacing $\sum_{n=1}^{r_i} I_{i\ell mn}^{(w)}/r_i$ with $h_{\ell, m}^{(w)}$. Under case (4), we do not need to generate additional $I_{i\ell m r_i'}^{(w)}$, for $r_i'=r_i+1,r_i+2,\ldots$. 

Figure \ref{fig:BeRF_Heuristic1} illustrates how feasibility decisions are concluded for two newly-added thresholds during the second pass, where $T_\ell^{(2)}=\{h_{\ell, 1}^{(2)}, h_{\ell, 2}^{(2)}\}$. The first pass is identical to what is shown in Figure \ref{fig:BeRF_Boundary}. After collecting $r_{i\ell 2}^{(1)}$ observations in the first pass, we compare the values of $v_{i\ell}^{\rm LB}$ and $v_{i\ell}^{\rm UB}$ at this point with the new thresholds $h_{\ell, 1}^{(2)}$ and $h_{\ell, 2}^{(2)}$. In this case, system $i$ is immediately declared infeasible for threshold $h_{\ell, 1}^{(2)}$ because $h_{\ell, 1}^{(2)}$ is smaller than $v_{i\ell}^{\rm LB}$. In Figure \ref{fig:BeRF_Heuristic1}, we let $r_{i\ell 1}^{(2)}=r_{i\ell 2}^{(1)}$ denote the number of observations to conclude its feasibility. To conclude feasibility decision for threshold $h_{\ell,2}^{(2)}$, we need to collect additional observations since $h_{\ell,2}^{(2)}$ lies between $v_{i\ell}^{\rm LB}$ and $v_{i\ell}^{\rm UB}$. In this example, after collecting $r_{i\ell 2}^{(2)}$ observations, system $i$ is declared feasible to threshold $h_{\ell,2}^{(2)}$.
Algorithm \ref{alg:berf_n} presents the statement of the $w$th pass procedure $\mathcal{BRF}_N^{(w)}$, where $w\geq 2$.
\begin{algorithm}[h!]
{\fontsize{10}{13}\selectfont	
    \begin{algorithmic}
    \caption{The $w$th Pass of Procedure Bernoulli Recycled Feasibility, $\mathcal{BRF}_{N}^{(w)}$, where $w\geq 2$}
    \label{alg:berf_n}
		\State {\bf [Setup:]} Same as in Algorithm \ref{alg:berf_b}. 
		\For{each system $i\in\Gamma$}
		\State [{\bf Initialization}:] 
		\begin{itemize}
            \item Set ${\rm ON}=L^{(w)}$ and ${\rm ON}_\ell=\{1,\ldots,d_\ell^{(w)}\}$ for $\ell\in {\rm ON}$. 
            \item Obtain $r_i, \bar{Y}_{i\ell}(r_i), {\rm LAST}_{i\ell}, v_{i\ell}^{\rm LB}, v_{i\ell}^{\rm UB}$, ${\rm SEED}^y_i$, and ${\rm SEED}^u_i$ for $i\in\Gamma$ and $\ell=1,\ldots,s$ from $\mathcal{BRF}^{(1)}$ (if $w=2$) or $\mathcal{BRF}_N^{(w-1)}$ (if $w\geq 3$). 
		\end{itemize}
		\State [{\bf Initial Feasibility Check} for $\mathcal{BRF}_N^{(w)}$:] 
        \For{each constraint $\ell\in{\rm ON}$}   	
        \For{each threshold $m\in {\rm ON}_\ell$}
        \State If $v_{i\ell}^{\rm UB}\leq h_{\ell,m}^{(w)}$ and $v_{i\ell}^{\rm LB}<h_{\ell,m}^{(w)}$, set $Z_{i\ell m}^{(w)}=1$ and ${\rm ON}_\ell={\rm ON}_\ell \setminus \{m\}$;
        \State Else if $v_{i\ell}^{\rm LB}\geq h_{\ell, m}^{(w)}$ and $v_{i\ell}^{\rm UB}>h_{\ell,m}^{(w)}$, set $Z_{i\ell m}^{(w)}=0$ and ${\rm ON}_\ell={\rm ON}_\ell \setminus \{m\}$;
        \State Else if $v_{i\ell}^{\rm UB}\leq h_{\ell, m}^{(w)}\leq v_{i\ell}^{\rm LB}$,
            \begin{itemize}[leftmargin=2.5cm]
                \item[] if ${\rm LAST}_{i\ell}={\rm UB}$, set $Z_{i\ell m}^{(w)}=0$ and ${\rm ON}_\ell={\rm ON}_\ell \setminus \{m\}$;
                \item[] if ${\rm LAST}_{i\ell}={\rm LB}$, set $Z_{i\ell m}^{(w)}=1$ and ${\rm ON}_\ell={\rm ON}_\ell \setminus \{m\}$.
            \end{itemize}
        \EndFor
        \State If ${\rm ON}_\ell=\emptyset$, set ${\rm ON}={\rm ON} \setminus \{\ell\}$.
    \EndFor
    \State [{\bf Stopping Condition} for $\mathcal{BRF}_N^{(w)}$:] Same as in the [{\bf Stopping Condition} for $\mathcal{BRF}_B^{(w)}$] from Algorithm \ref{alg:berf_b} except that we do not generate $U_{ir_i}\sim U(0,1)$. Go to [{\bf Feasibility Check} for $\mathcal{BRF}_N^{(w)}$] after the execution.

    \State [{\bf Feasibility Check} for $\mathcal{BRF}_N^{(w)}$:] 
    \For{each constraint $\ell\in{\rm ON}$}   	\State Set $v_{i\ell}^{\rm LB}=\max\left\{v_{i\ell}^{\rm LB}, \bar{Y}_{i\ell}(r_i)-H_\ell/r_i\right\}$. If $v_{i\ell}^{\rm LB}$ is updated, set ${\rm LAST}_{i\ell}={\rm LB}$.
        \State Set $v_{i\ell}^{\rm UB}=\min\left\{v_{i\ell}^{\rm UB}, \bar{Y}_{i\ell}(r_i)+H_\ell/r_i\right\}$. If $v_{i\ell}^{\rm UB}$ is updated, set ${\rm LAST}_{i\ell}={\rm UB}$.
        \For{each threshold $m\in {\rm ON}_\ell$}
        \State If $v_{i\ell}^{\rm UB}\leq h_{\ell,m}^{(w)}$, set $Z_{i\ell m}^{(w)}=1$ and ${\rm ON}_\ell={\rm ON}_\ell \setminus \{m\}$;
        \State If $v_{i\ell}^{\rm LB}\geq h_{\ell, m}^{(w)}$, set $Z_{i\ell m}^{(w)}=0$ and ${\rm ON}_\ell={\rm ON}_\ell \setminus \{m\}$.
        \EndFor
        \State If ${\rm ON}_\ell=\emptyset$, set ${\rm ON}={\rm ON}\setminus \{\ell\}$.
		\EndFor
        \State Go to [{\bf Stopping Condition} for $\mathcal{BRF}_N^{(w)}$].
		\EndFor%
	\end{algorithmic}
 }
\end{algorithm}

Note that the two heuristic versions $\mathcal{BRF}_B^{(w)}$ and $\mathcal{BRF}_N^{(w)}$ may differ in the number of observations required to conclude feasibility decisions. In addition, $\mathcal{BRF}_B^{(w)}$ will require more computational effort in generating the dummy Bernoulli variables $I_{i\ell mn}^{(w)}$ (although this is expected to be much smaller than for collecting the observations $Y_{i\ell n}$). 
 In general, our experiments show that $\mathcal{BRF}_B^{(w)}$ tends to require fewer observations compared with $\mathcal{BRF}_N^{(w)}$ (see Table \ref{tab:berfwsingle1} in Section \ref{subsec:MBeRFExp}). Although we do not prove the statistical validity of Algorithms  \ref{alg:berf_b} and \ref{alg:berf_n}, we do not observe the validity being violated in our experiments (see Section \ref{subsubsec:HeuristicValidity}).

\paragraph{Heuristic approach $\mathcal{BRF}_{BN}^{(w)}$.}
$\mathcal{BRF}_{BN}$ combines $\mathcal{BRF}_B$ and $\mathcal{BRF}_N$. Specifically, for a newly-added threshold $h_{\ell, m}^{(w)}$, $\mathcal{BRF}_{BN}$ first 
performs feasibility check by 
comparing $v_{i\ell}^{\rm LB}$ and $v_{i\ell}^{\rm U
B}$ with respect to $h_{\ell,m}^{(w)}$ (as in Algorithm \ref{alg:berf_n}) and then
with respect to $\sum_{n=1}^{r_i} I_{i\ell m n}^{(w)}/r_i$ (as in Algorithm \ref{alg:berf_b}).
As $\mathcal{BRF}_{BN}^{(w)}$ makes a feasibility decision for $h_{\ell, m}^{(w)}$ when either $\mathcal{BRF}_B^{(w)}$ or $\mathcal{BRF}_N^{(w)}$ concludes its feasibility check, $\mathcal{BRF}_{BN}$ is guaranteed to perform no worse than both $\mathcal{BRF}_{B}$ and $\mathcal{BRF}_{N}$.  
Algorithm \ref{alg:berf_bn} shows the statement of the $w$th pass $\mathcal{BRF}_{BN}^{(w)}$, where $w\geq 2$. 

\begin{algorithm}[h!]
{\fontsize{10}{13}\selectfont	
    \begin{algorithmic}
    \caption{The $w$th Pass of Procedure Bernoulli Recycled Feasibility, $\mathcal{BRF}_{BN}^{(w)}$, where $w\geq 2$}
    \label{alg:berf_bn}
		\State {\bf [Setup:]} Same as in Algorithm \ref{alg:berf_b}.
		\For{each system $i\in\Gamma$}
		\State [{\bf Initialization}:] 
		Same as in $\mathcal{BRF}_B^{(w)}$ in Algorithm \ref{alg:berf_b} except that we obtain $r_i$, $\bar{Y}_{i\ell}(r_i)$, ${\rm LAST}_{i\ell}$, $v_{i\ell}^{\rm LB}$, $v_{i\ell}^{\rm UB}$, ${\rm SEED}_i^y$, and ${\rm SEED}_i^u$ for $i\in \Gamma$ and $\ell=1,\ldots,s$ from ${\cal BRF}_{BN}^{(w-1)}$ if $w\geq 3$.
        \State [{\bf Initial Feasibility Check} for $\mathcal{BRF}_{BN}^{(w)}$:]
        \State Perform [{\bf Initial Feasibility Check} of $\mathcal{BRF}_N^{(w)}$] as in Algorithm \ref{alg:berf_n}. 
        \State Perform [{\bf Initial Feasibility Check} of $\mathcal{BRF}_B^{(w)}$] as in Algorithm \ref{alg:berf_b}. 
    \State [{\bf Stopping Condition} for $\mathcal{BRF}_{BN}^{(w)}$:] Same as in the [{\bf Stopping Condition} for $\mathcal{BRF}_B^{(w)}$] in Algorithm \ref{alg:berf_b} and go to [{\bf Feasibility Check} for $\mathcal{BRF}_{BN}^{(w)}$] after the execution.

    \State [{\bf Feasibility Check} for $\mathcal{BRF}_{BN}^{(w)}$:] 
    \For{each constraint $\ell\in{\rm ON}$}   	\State Set $v_{i\ell}^{\rm LB}=\max\left\{v_{i\ell}^{\rm LB}, \bar{Y}_{i\ell}(r_i)-H_\ell/r_i\right\}$. If $v_{i\ell}^{\rm LB}$ is updated, set ${\rm LAST}_{i\ell}={\rm LB}$.
        \State Set $v_{i\ell}^{\rm UB}=\min\left\{v_{i\ell}^{\rm UB}, \bar{Y}_{i\ell}(r_i)+H_\ell/r_i\right\}$. If $v_{i\ell}^{\rm UB}$ is updated, set ${\rm LAST}_{i\ell}={\rm UB}$.
        \For{each threshold $m\in {\rm ON}_\ell$}
        \State If $v_{i\ell}^{\rm UB}\leq h_{\ell,m}^{(w)}$, set $Z_{i\ell m}^{(w)}=1$ and ${\rm ON}_\ell={\rm ON}_\ell \setminus \{m\}$;
        \State Else if $v_{i\ell}^{\rm LB}\geq h_{\ell, m}^{(w)}$, set $Z_{i\ell m}^{(w)}=0$ and ${\rm ON}_\ell={\rm ON}_\ell \setminus \{m\}$;
	\State Else, 
    \State \hspace{0.8cm} Set $I_{i \ell m r_i}^{(w)}=1$ if $U_{i r_i} \le h_{\ell, m}^{(w)}$ and 0 otherwise.
            \begin{itemize}[leftmargin=2.5  cm]
                \item[] If $v_{i\ell}^{\rm UB}\leq \sum_{n=1}^{r_i} I_{i \ell m n}^{(w)}/r_i$, set $Z_{i\ell m}^{(w)}=1$ and ${\rm ON}_\ell={\rm ON}_\ell \setminus \{m\}$;
                \item[] If $v_{i\ell}^{\rm LB}\geq \sum_{n=1}^{r_i} I_{i \ell m n}^{(w)}/r_i$, set $Z_{i\ell m}^{(w)}=0$ and ${\rm ON}_\ell={\rm ON}_\ell \setminus \{m\}$.
            \end{itemize}
        \EndFor
        \State If ${\rm ON}_\ell=\emptyset$, set ${\rm ON}={\rm ON}\setminus \{\ell\}$.
		\EndFor
        \State Go to [{\bf Stopping Condition} for $\mathcal{BRF}_{BN}^{(w)}$].

	\EndFor%
	\end{algorithmic}
 }
\end{algorithm}

\section{Statistical Validity} 
\label{sec:statistical validity}

In this section, we prove the statistical validity of the first-pass of $\mathcal{MPB}$, i.e., $\mathcal{BRF}^{(1)}$. Section~\ref{subsec:oneone} proves the result for a single system and Section~\ref{subsec:manymany} generalizes the proof for multiple systems.  

\subsection{Single System} 
\label{subsec:oneone}

First, we consider one system $i$, one constraint $\ell$, and one threshold $h_{\ell,m}^{(1)}$ during the first pass. 
For simplicity, we drop subscripts $i, m$, and superscript $(1)$ from $p_{i\ell}$, ${\rm CD}_{i\ell}$, $Y_{i\ell n}$, $U_{in}$, $h_{\ell, m}^{(1)}$, and $I_{i \ell mn}^{(1)}$.

We begin with Lemma \ref{thm:single_system_single_threshold}, which shows the statistical validity of $\mathcal{BRF}^{(1)}$ for the case where a single system, a single constraint, and one particular threshold are considered. 

\begin{lemma} 
\label{thm:single_system_single_threshold}
For a single system, a constraint $\ell$ with threshold $h$, $\mathcal{BRF}^{(1)}$ guarantees $\Pr\left( {\rm CD}_\ell(h_\ell) \right)  \ge 1 - \beta_\ell$.
\end{lemma}

\begin{proof}
    \renewcommand{\qedsymbol}{}
    The main idea of our procedure is to compare the Bernoulli observations from the system on the constraint, which have a probability $p_\ell$, with Bernoulli data from a dummy system that has a probability $h_\ell$. The comparison between the two systems is designed similar to Bernoulli selection procedures such as \cite{tamhane1985some}.

    The simple random walk process $S_\ell(r) = \sum_{n=1}^{r} \left(Y_{\ell n} - I_{\ell n}\right)$ tracks the cumulative difference in observations.
The state transitions depend on the outcomes of $Y_{\ell n}$ and $I_{\ell n}$. Specifically, when the current state of $S_\ell(r)$ is $a$, it transits to $a+1$ when $Y_{\ell n}$ is 1 and $I_{\ell n}$ is 0 and transits to $a-1$ when $Y_{\ell n}$ is 0 and $I_{\ell n}$ is 1. Furthermore, it stays at $a$ when $Y_{\ell n}$ and $I_{\ell n}$ are identical. Consequently, the probabilities of transiting from $a$ to $a+1$ and from $a$ to $a-1$ are given by $p_\ell(1-h_\ell)$ and $(1-p_\ell)h_\ell$, respectively, and the probability of staying at $a$ is $p_\ell h_\ell+(1-p_\ell)(1-h_\ell)$. 



For our analysis, we focus on the slippage configuration (SC), which means $
{(1-p_\ell)h_\ell \over p_\ell(1-h_\ell)} = \theta_\ell$ or $
{p_\ell (1-h_\ell) \over (1-p_\ell) h_\ell} = \theta_\ell$, depending on whether the system is desirable or unacceptable (see Section \ref{subsec:CorrectDecision}). The SC is expected to be the least-favorable configuration and we discuss the probability of correct decision when the system is desirable or unacceptable separately. 

By Section 3.6.1 of \citet{TaylorKarlin1994}, the lower absorption probability of $S_\ell(r)$ to $-H_\ell$ (rather than $H_\ell$) can be calculated as  
\begin{align}
u_{H_\ell}
&=
\left\{
\begin{array}{ll}
0.5 
& \quad \text{if } \frac{(1-p_\ell)h_\ell}{p_\ell (1-h_\ell)} = 1, \\[6pt]
\dfrac{\left(\frac{(1-p_\ell)h_\ell}{p_\ell (1-h_\ell)}\right)^{H_\ell}- \left(  \frac{(1-p_\ell) h_\ell}{p_\ell (1-h_\ell)}\right)^{2 H_\ell} }{1 - \left( \frac{(1-p_\ell)h_\ell}{p_\ell (1-h_\ell)} \right)^{2H_\ell}} 
& \quad \text{if }  \frac{(1-p_\ell) h_\ell}{p_\ell (1-h_\ell)} \ne 1,
\end{array}
\right\}
=
\dfrac{\left(  \frac{(1-p_\ell) h_\ell}{p_\ell (1-h_\ell)}\right)^{H_\ell}}{1+\left(  \frac{(1-p_\ell) h_\ell}{p_\ell (1-h_\ell)}\right)^{H_\ell}}.
\label{eqn:AbsorptionProb}
\end{align}
Note that $u_{H_\ell}$ from Equation \eqref{eqn:AbsorptionProb} is increasing in $\frac{(1-p_\ell)h_\ell}{p_\ell(1-h_\ell)}$.
When the system is desirable, i.e., $
{(1-p_\ell) h_\ell \over p_\ell (1-h_\ell)} \ge \theta_\ell$, the probability of correct decision corresponds to the lower absorption probability to $-H_\ell$, which can be calculated following Equation (\ref{eqn:AbsorptionProb}) as at least
${\theta^{H_\ell} \over 1 + \theta^{H_\ell}}$.  
Similarly, when the system is unacceptable, i.e., $
{p_\ell (1-h_\ell) \over (1-p_\ell)h_\ell} \ge \theta_\ell$, 
the probability of correct decision is the upper absorption probability to $H$, which from Equation (\ref{eqn:AbsorptionProb}) is at least $1-\frac{\theta^{-H_\ell}}{1+\theta^{-H_\ell}}={\theta^{H_\ell} \over 1 + \theta^{H_\ell}}$.
When the system is acceptable, any decision is a correct decision. By putting all three cases together, we get 
\[
\Pr\left( {\rm CD}_\ell(h_\ell) \right)  \ge {\theta^{H_\ell} \over 1 + \theta^{H_\ell}} \ge 1 - \beta_\ell, 
\]
where the second inequality holds because $H_\ell$ satisfies \eqref{eqn:H}. 
\hspace{6cm} $\square$
\end{proof}
\vspace{-0.5cm}

Now, we prove the statistical validity of $\mathcal{BRF}^{(1)}$ for a single system with thresholds $h_{\ell,1}^{(1)} < h_{\ell,2}^{(1)} < \cdots < h_{\ell, d_\ell^{(1)}}^{(1)}$ on constraint $\ell$ in Lemma \ref{thm:single_system_multiple_threshold}. We add the subscript $m$ back to all notation while still dropping the subscript $i$ and superscript $(1)$. 

\begin{lemma} 
\label{thm:single_system_multiple_threshold}
For a single system and a single constraint with thresholds $h_{\ell,1},h_{\ell,2},\ldots,h_{\ell, d_\ell}$, where $d_\ell\geq 1$, $\mathcal{BRF}^{(1)}$ guarantees 
\[
\Pr\left( {\rm CD}_\ell \right) = \Pr\left( \cap_{m=1}^{d_\ell}{\rm CD}_\ell(h_{\ell, m}) \right)  \ge \left\{ \begin{array}{ll} 
1 - \beta_\ell, & \mbox{if } d_\ell=1; \\
1 - 2\beta_\ell, & \mbox{if } d_\ell\ge 2.
\end{array} \right.
\]
\end{lemma}

\begin{proof}
    \renewcommand{\qedsymbol}{}
    If $d_\ell=1$, then 
\[
\Pr\left( {\rm CD}_\ell \right) = \Pr\left( \cap_{m=1}^{d_\ell}{\rm CD}_{\ell}(h_{\ell, m}) \right)  = \Pr\left( {\rm CD}(h_{\ell, 1}) \right) \ge 1 - \beta_\ell,
\]
where the inequality holds due to Lemma~\ref{thm:single_system_single_threshold}.

If $d_\ell \ge 2$, there are three possible cases regarding the feasibility of the system. 
\begin{itemize}
    \item[] {\bf Case 1:}  The system is desirable with respect to all thresholds. Since $h_{\ell,1} < h_{\ell,2} < \cdots < h_{\ell,d_\ell}$ and we use the same random number $U_{n}$ to generate the Bernoulli random variables $I_{\ell mn}$ for $m=1,2,\ldots,d_\ell$, we have $I_{\ell 1n} \le I_{\ell 2n} \le \cdots \le I_{\ell d_\ell n}$, for $n=1,2,\ldots$. It follows that
    \begin{equation} \label{eqn:dummy}
    \sum_{n=1}^{r}\left(Y_{\ell n} -I_{\ell 1n}\right) \ge \sum_{n=1}^{r}\left(Y_{\ell n} -I_{\ell 2n}\right)\ge \cdots \ge \sum_{n=1}^{r}\left(Y_{\ell n} -I_{\ell d n}\right).
    \end{equation}
    A correct decision with respect to $h_{\ell, m}$ is made if $\sum_{n=1}^{r}(Y_{\ell n} -I_{\ell mn})$ reaches $-H_\ell$ before $H_\ell$. Equation \eqref{eqn:dummy} implies that if we have a correct decision at $h_{\ell, 1}$, i.e.,  if $\sum_{n=1}^{r}(Y_{\ell n} -I_{\ell 1n})$ reaches $-H_\ell$ before $H_\ell$, then we also have correct decisions for the other thresholds $h_{\ell, 2}< \cdots < h_{\ell, d_\ell}$. That is, we have ${\rm CD}(h_{\ell,1}) \subseteq {\rm CD}(h_{\ell, 2}) \subseteq \cdots \subseteq {\rm CD}(h_{\ell, d_\ell})$. Therefore, Lemma~\ref{thm:single_system_single_threshold} yields
    \[ 
    \Pr\left( {\rm CD}_\ell \right)=\Pr\Big(\cap_{m=1}^{d_\ell} {\rm CD}_\ell(h_{\ell, m})\Big) = \Pr\Big( {\rm CD}_\ell(h_{\ell,1})\Big) \ge 1 - \beta_\ell \ge 1 - 2 \beta_\ell.
    \] 
    \item[] {\bf Case 2:} The system is unacceptable with respect to all thresholds. A correct decision occurs when $\sum_{n=1}^{r}(Y_{\ell n} -I_{\ell mn})$ reaches $H_\ell$ before $-H_\ell$. Similarly to {\bf Case 1}, if a correct decision is made for $h_{\ell, d_\ell}$, correct decisions are also made for thresholds $h_{\ell, 1}\leq \cdots\leq h_{\ell, d_\ell}$, i.e., ${\rm CD}_\ell(h_{\ell, d_\ell}) \subseteq {\rm CD}_\ell(h_{\ell,d_\ell-1}) \subseteq \cdots \subseteq {\rm CD}_\ell(h_{\ell,1})$. Thus, it follows from Lemma~\ref{thm:single_system_single_threshold} that
    \[ 
    \Pr\left( {\rm CD}_\ell \right)=\Pr\Big(\cap_{m=1}^{d_\ell} {\rm CD}_\ell(h_{\ell,m})\Big) = \Pr\Big( {\rm CD}_\ell(h_{\ell,d_\ell})\Big) \ge 1 - \beta_\ell \ge 1 - 2\beta_\ell.
    \]
    \item[] {\bf Case 3:} The system is unacceptable with respect to thresholds $h_{\ell, 1} < \cdots < h_{\ell, \underline{\kappa}}$ and is desirable with respect to thresholds $h_{\ell, \overline{\kappa}} < \cdots < h_{\ell,d_\ell}$, where $1\leq \underline{\kappa} <  \overline{\kappa}\leq d_\ell$.
Suppose that we add two additional thresholds $h_\ell^L$ and $h_\ell^U$ such that 
\begin{equation} 
\label{eqn:LBUB}
\frac{h_\ell^U (1-p_\ell)}{(1-h_\ell^U)p_\ell}  = \theta_\ell \quad \mbox{and} \quad {(1-h_\ell^L)p_\ell \over h_\ell^L(1-p_\ell)} =\theta_\ell,
\end{equation}
i.e., these thresholds have odds-ratios with respect to $p_\ell$ that exactly equal $\theta_\ell$. Then we have that the system is desirable for $h_\ell^U$, unacceptable for $h_\ell^L$, and 
\[
h_{\ell, 1} < \cdots < h_{\ell, \underline{\kappa}} \le h_\ell^L < h_{\ell,\underline{\kappa}+1} < \cdots < h_{\ell,\overline{\kappa}-1} < h_\ell^U \le h_{\ell,\overline{\kappa}} < \cdots < h_{\ell,d_\ell}.
\]
A correct decision occurs when the system is deemed infeasible with respect to $h_{\ell,1}, \ldots, h_{\ell,\underline{\kappa}}$ but feasible with respect to $h_{\ell,\overline{\kappa}}, \ldots, h_{\ell,d_\ell}$. Consider dummy Bernoulli random variables $I_{\ell n}^{L}$ and $I_{\ell n}^{U}$ that are generated based on thresholds $h_\ell^{L}$ and $h_\ell^{U}$, respectively, applying the same random number $U_{n}$ used for generating $I_{\ell 1n},\ldots, I_{\ell d_\ell n}$, which yields $I_{\ell 1n}\leq \cdots\leq I_{\ell \underline{\kappa} n}\leq I_{\ell n}^L\leq I_{\ell (\underline{\kappa}+1) n}\leq \cdots \leq I_{\ell (\overline{\kappa}-1) n}\leq I_{\ell n}^U \leq I_{\ell \overline{\kappa} n}\leq \cdots \leq I_{\ell d_\ell n}$ for $n=1,2,\ldots$.
It follows that ${\rm CD}_\ell(h_\ell^{L}) \subseteq {\rm CD}_\ell(h_{\ell, \underline{\kappa}})  \subseteq \cdots  \subseteq {\rm CD}_\ell(h_{\ell,1})$ and ${\rm CD}_\ell(h_\ell^{U}) \subseteq {\rm CD}_\ell(h_{\ell,\overline{\kappa}}) \subseteq \cdots \subseteq  {\rm CD}_{\ell}(h_{\ell,d_\ell})$. For thresholds $h_{\ell,\underline{\kappa}+1}, \ldots, h_{ \ell, \overline{\kappa}-1}$, the system is acceptable, meaning that any feasibility decision is a correct decision. Therefore, the Bonferroni inequality and Lemma \ref{thm:single_system_single_threshold} yield
\begin{align*}
\Pr\left( {\rm CD}_\ell \right)  &= \Pr\Big(\cap_{m=1}^{d_\ell} {\rm CD}_\ell(h_{\ell, m})\Big) 
\geq \Pr\Big( {\rm CD}_\ell(h_\ell^{L}) \cap {\rm CD}_\ell(h_\ell^{U})\Big) \\
&\ge \Pr\Big( {\rm CD}_\ell(h_\ell^{L})\Big) + \Pr\Big({\rm CD}_\ell(h_\ell^{U})\Big) -1 \ge 1 - 2\beta_\ell. \tag*{$\square$}
\end{align*}
\end{itemize}
\end{proof}
\vspace{-1.3cm}
\begin{remark}
\label{remark:SlippageConfig}
    Due to Equation \eqref{eqn:LBUB}, the two ``most difficult'' thresholds for an odds-ratio $\theta_\ell>1$ on constraint $\ell$ with probability $p_{\ell}$ can be expressed with the two functions 
    \begin{equation} 
\label{eqn:thresholds}
f_{\ell}^{L}(p_{\ell}, \theta_\ell) = \frac {p_{\ell}}{p_{\ell} + (1-p_{\ell})\theta_\ell} \quad \mbox{and} \quad f_{\ell}^{U}(p_{\ell}, \theta_\ell) = \frac{p_{\ell}\theta_\ell}{p_{\ell} (\theta_\ell-1) + 1}
\end{equation} 
that output the threshold that system $i$ is at the boundary of being desirable and unacceptable, respectively. Although $f_\ell^L(p_{\ell},\theta_\ell)$ and $f_\ell^U(p_{\ell},\theta_\ell)$ both produce the ``most difficult'' threshold values for the constraint with an odds-ratio $\theta_\ell$, they are not symmetric in terms of $p_{\ell}$. For example, when $p_{\ell}=0.15$, we have $f_\ell^L(0.15, 1.2)=0.1282$ and $f_\ell^U(0.15, 1.2)=0.1748$, and hence the threshold determined by $f_\ell^L(p_{\ell}, \theta_\ell)$ is closer to $p_{\ell}=0.15$ than that determined by $f_\ell^U(p_{\ell}, \theta_\ell)$. 
In particular, due to the definition of the odds-ratio IZ parameter, the threshold determined by $f_\ell^L(p_{\ell}, \theta_\ell)$ is closer (further) from $p_{\ell}$ compared with the threshold determined by $f_\ell^U(p_{\ell}, \theta_\ell)$ when $p_{\ell}$ is smaller (larger) than 0.5.
\end{remark}

Finally, we prove the statistical validity of $\mathcal{BRF}^{(1)}$ for a single system with $s$ constraints in Theorem \ref{thm:single_system_multiple}. We add the superscript (1) back to all notation, while still dropping the subscript $i$. 
\begin{theorem}
\label{thm:single_system_multiple}
For a single system and $s$ constraints with thresholds $h_{\ell,1}^{(1)},h_{\ell,2}^{(1)},\ldots,h_{\ell, d_\ell^{(1)}}^{(1)}$, where $\ell=1,2,\ldots, s$, $\mathcal{BRF}^{(1)}$ guarantees $\Pr (\cap_{\ell=1}^s {\rm CD}_\ell)  \ge 1 - \beta.$
\end{theorem}
\begin{proof}

\renewcommand{\qedsymbol}{}

Consider $\beta_\ell$ determined by $(i)$ in Equation \eqref{eqn:betaell}. 
Then, from Lemma \ref{thm:single_system_multiple_threshold}, $\Pr\left( {\rm CD}_{\ell} \right) \ge 1 - \beta_\ell = 1 - \beta/s$  for constraint $\ell$ with $d_\ell^{(1)}=1$ and 
$\Pr\left( {\rm CD}_{\ell} \right)  \ge 1 - 2 \beta_\ell = 1 - 2 (\beta/(2s)) = 1 - \beta/s$ for constraint $\ell$ with $d_\ell >1$. Thus, it follows from the Bonferroni inequality that 
\[
\Pr\left( \cap_{\ell=1}^s {\rm CD}_{\ell} \right) \ge 1 - \sum_{\ell=1}^s \left(1 - \Pr\left( {\rm CD}_\ell \right) \right) \ge 1 - \sum_{\ell=1}^s {\beta \over s} = 1 - \beta.
\]

Next, we consider $\beta_\ell$ as in $(ii)$ in Equation \eqref{eqn:betaell}. Let $s_1$ ($s_2$) be the number of constraints with one threshold (two or more thresholds). Then $\beta_\ell = \beta/D$ where $D=s_1+2s_2$. From Lemma \ref{thm:single_system_multiple_threshold}, for constraint $\ell$ with $d_\ell=1$, we have $\Pr\left( {\rm CD}_{\ell} \right) \ge 1 - \beta/D$, and for constraint $\ell$ with $d_\ell >1$, we have  
$\Pr\left( {\rm CD}_{\ell} \right)  \ge 1 - 2 \beta_\ell = 1 - 2 \beta/D$. Thus, the Bonferroni inequality yields
\begin{align*}
    \Pr\left( \cap_{\ell=1}^s {\rm CD}_{\ell} \right) \ge 1 - \sum_{\ell=1}^s \left(1 - \Pr\left( {\rm CD}_\ell \right)\right) \ge 1 -  s_1 {\beta \over D} - 2 s_2 {\beta \over D} = 1 - {(s_1 + 2s_2) \beta \over D } = 1 - \beta. 
    \tag*{$\square$}
\end{align*}
\end{proof}
\vspace{-0.8cm}
As discussed in Remark \ref{remark:berf}, if guaranteeing statistical validity is desired, the decision maker is advised to conduct the feasibility check for all possible thresholds in a single pass, i.e., by performing $\mathcal{BRF}$. We present the following corollary for the statistical validity of $\mathcal{BRF}$. The proof of Corollary \ref{cor:berf} is identical to the proof of Theorem \ref{thm:single_system_multiple} except that $d_\ell^{(1)}$ includes all thresholds. 
\begin{corollary}
\label{cor:berf}
    For a single system and $s$ constraints with thresholds $h_{\ell,1}^{(1)}, h_{\ell, 2}^{(1)}, \ldots, h_{\ell,d_\ell^{(1)}}^{(1)}$ for $\ell=1,2,\ldots,s$, $\mathcal{BRF}$ guarantees $\Pr (\cap_{\ell=1}^s {\rm CD}_\ell)\geq 1-\beta$.
\end{corollary}

\subsection{Multiple systems}
\label{subsec:manymany}
Now we consider $k\ge 2$ systems. Theorem \ref{thm:BeRFStatisticalValidityMultiSys} proves the statistical validity of $\mathcal{BRF}^{(1)}$ in the context of multiple systems, multiple constraints, and multiple thresholds.

\begin{theorem}
\label{thm:BeRFStatisticalValidityMultiSys}
    $\mathcal{BRF}^{(1)}$ guarantees ${\rm PCD} \ge 1 - \alpha.$
\end{theorem}
\begin{proof}
\renewcommand{\qedsymbol}{}
Since there are multiple systems, the nominal probability of error is given in Equation (\ref{eqn:beta}).

First, when systems are simulated independently, it follows from Theorem \ref{thm:single_system_single_threshold} that 
\begin{equation*}
{\rm PCD} =\Pr\left(\cap_{i=1}^k \cap_{\ell=1}^s {\rm CD}_{i\ell}\right)\\
= \prod_{i=1}^k \Pr\left(\cap_{\ell=1}^{s} {\rm CD}_{i\ell}\right) \\
\ge  (1 - \beta)^k = 1 - \alpha.
\end{equation*}

Next, when CRN is employed, the Bonferroni inequality and Theorem~\ref{thm:single_system_multiple} yield
\begin{align*}
    {\rm PCD} =\Pr\left(\cap_{i=1}^k \cap_{\ell=1}^s {\rm CD}_{i\ell}\right)\ge 1 - \sum_{i=1}^k \left( 1 - \Pr\left(\cap_{\ell=1}^{s} {\rm CD}_{i\ell}\right) \right)\ge 1 - \sum_{i=1}^k \beta = 1- k {\alpha \over k} = 1 - \alpha.
\tag*{$\square$}
\end{align*}
\end{proof}
\vspace{-0.8cm}

Finally, we present Corollary \ref{cor:berf_multisys} to show the statistical validity of $\mathcal{BRF}$. The proof of Corollary \ref{cor:berf_multisys} is straightforward due to Corollary \ref{cor:berf} and the same techniques as in the proof of Theorem \ref{thm:BeRFStatisticalValidityMultiSys}.
\begin{corollary}
\label{cor:berf_multisys}
    $\mathcal{BRF}$ guarantees ${\rm PCD}\geq 1-\alpha$.
\end{corollary}

\section{Experiments}
\label{sec:experiments}

In this section, we demonstrate the performance of our proposed procedures. We consider three procedures: (1) $\mathcal{BRF}$, our single-pass feasibility check procedure for subjective probability constraints that includes all possible thresholds up front; (2) $\mathcal{RF}$, a single-pass feasibility check procedure for subjective normal constraints due to \citet{zhou2022finding}; and (3) $\mathcal{MPB}$, our multi-pass feasibility check procedure for subjective probability constraints. 
We first discuss the stopping time of $\mathcal{BRF}$ (due to its unbounded continuation region) in Section \ref{subsec:BRFStoppingTime}. 
We then perform a numerical evaluation of $\mathcal{BRF}$ in comparison to $\mathcal{RF}$ in Section \ref{subsec:RFComparison}. Subsequently, we provide numerical results for the multi-pass procedure $\mathcal{MPB}$ in Section \ref{subsec:MBeRFExp}. Finally, we demonstrate the performance of all three procedures through the inventory example from Section \ref{sec:Intro} in Section \ref{subsec:InventoryExp}. 

All experiments are conducted with nominal error probability $\alpha = 0.05$. For procedures $\mathcal{BRF}$ and $\mathcal{MPB}$, we set IZ $\theta_\ell\in \{1.2, 1.5\}$ for each constraint $\ell=1,\ldots,s$. 
We report the estimated probability of correct decision (PCD) and the average total number of observations (OBS). We run each experiment for 10,000 macro-replications and the estimated PCD values are meaningful up to the 0.001th digit while the OBS values are meaningful up to the first three digits (except for Section \ref{subsec:InventoryExp} where we run 1,000 macro-replication and the standard errors are provided).

\subsection{Stopping Time of $\mathcal{BRF}$}
\label{subsec:BRFStoppingTime}


As the continuation region of $\mathcal{BRF}$ for constraint $\ell$ (i.e., $(-H_\ell, H_\ell)$) is unbounded, there is no finite upper bound on its stopping time. In this section, we analyze the stopping time of $\mathcal{BRF}$. For simplicity, we consider one system, one constraint, and one threshold through the first pass; therefore, we drop the subscripts $i, \ell, m$, and superscript $(1)$ in the remainder of this section. 

\paragraph{Theoretical expected stopping time.}
We begin with an analytical expression for the expected stopping time. Recall that the monitoring statistic $\sum_{n=1}^{r} ( Y_n-I_{n})$ follows a simple random walk with transition probabilities discussed in the proof of Lemma \ref{thm:single_system_single_threshold}. 
Referring to Section 3.6.1 in \citet{TaylorKarlin1994}, consider the classical gambler’s ruin problem, starting with initial balance $H$. The expected number of trials $v_H$ until the process reaches either $0$ or $2H$ is:
\begin{equation}
\label{eqn:ExpectedNumSteps}
    v_H = \left\{ \begin{array}{ll}
            \frac{H^2}{2p(1-h)}, & \mbox{ if $p=h$ };\\
            {H \over p-h}\left[2\left({{1-{{\left((1-p)h \over p (1-h) \right)}^H}}\over{{1-{{\left((1-p)h \over p(1-h) \right)}^{2H}}}}}\right)-1\right], & \mbox{ otherwise}.
        \end{array} \right.
\end{equation}

Under the SC where $\frac{(1-p)h}{p(1-h)}=\theta$ when $p<h$ or $\frac{p(1-h)}{(1-p)h}=\theta$ when $p>h$, we have
\begin{align*}
    v_H=\frac{H}{p-h}\left[ 2 \left(\frac{1-\theta^{H}}{1-\theta^{2H}}\right)-1 \right] &= \begin{cases}
        \frac{H}{p-h}\times \frac{1-\theta^{H}}{1+\theta^{H}} &\quad \text{ when }p<h; \\
        \frac{H}{p-h}\times \frac{\theta^{H}-1}{\theta^{H}+1} &\quad \text{ when }p>h.
    \end{cases}
\end{align*}
The expected stopping time is longest when the probability $p$ and threshold $h$ coincide. In that case, there is no directional drift -- the random walk is equally likely to move up or down at each step. With no bias toward either boundary, the process must rely purely on random fluctuations to reach absorption, which increases the expected stopping time. In contrast, when $p$ and $h$ differ, a nonzero drift emerges and the trajectory moves more consistently toward one boundary, reducing the expected duration. We now aim to better understand the expected stopping time under general scenarios (i.e., $p\not=h$) and the ``worst-case'' scenario (i.e., $p=h$) through experiments.

\paragraph{Empirical stopping time under general scenarios.}

We next provide the average empirical stopping times for different values of $p,h$, and $\theta$, where we set $\alpha=\beta=\beta_1=0.05$ as we consider a single system, constraint, and threshold. Specifically, we consider $\theta\in \{1.2, 1.5\}$ and set $p\in \{0.15, 0.5\}$. For each case, $h$ is chosen so that $\frac{p(1-h)}{(1-p)h}\in \{1, \theta, 2\theta, 5\theta, 10\theta \}$ (and hence the system is unacceptable with respect to $h$). Note that for a fixed $\theta$, larger values of $\frac{(1-p)h}{p(1-h)}$ indicate an easier problem. Table \ref{tab:EmpiricalStopping_General} presents the results, where we also report the theoretical value in Equation \eqref{eqn:ExpectedNumSteps} and the threshold $h$. Note that $H=17$ and 8 when $\theta=1.2$ and $1.5$, respectively. 

\begin{table}[h!]
    \centering
    \renewcommand{\arraystretch}{0.8}
    \begin{tabular}{c||c|ccc||ccc}
        \toprule
         & & \multicolumn{3}{c||}{$p=0.15$} & \multicolumn{3}{c}{$p=0.5$} \\   \midrule
        & $\frac{p(1-h)}{(1-p)h}$ & $h$ & Theoretical & Empirical (s.e) & $h$ & Theoretical & Empirical (s.e.)  \\  \hline
        \multirow{ 5}{*}{$\theta=1.2$} & 1 & 0.150 & 1133.333 & 1141.074 (9.259) & 0.500 & 578.000 & 577.215 (4.726) \\
        & $\theta$ & 0.128 & 712.718 & 718.301 (5.109) & 0.455 & 341.739 & 339.158 (2.440) \\
         & $2\theta$ & 0.069 & 208.571 & 209.060 (0.775) & 0.294 & 82.571 & 82.391 (0.298) \\
         & $5\theta$ & 0.029 & 140.000 & 140.322 (0.386) & 0.143 & 47.600 & 47.733 (0.120) \\  
         & $10\theta$ & 0.015 & 125.455 & 125.730 (0.311) & 0.077 & 40.182 & 40.246 (0.084) \\ \hline
         \multirow{ 5}{*}{$\theta=1.5$} & 1 & 0.150 & 250.980 & 251.392 (2.050) & 0.500 & 128.000 & 128.096 (1.037) \\
         & $\theta$ & 0.105 & 165.393 & 166.116 (1.178) & 0.400 & 73.991 & 73.088 (0.519) \\
         & $2\theta$ & 0.056 & 84.680 & 85.301 (0.418) & 0.250 & 31.990 & 32.014 (0.149)  \\
         & $5\theta$ & 0.023 & 62.986 & 62.920 (0.244) & 0.118 & 20.923 & 20.859 (0.071) \\
         & $10\theta$ & 0.012 & 57.815 & 57.906 (0.205) & 0.063 & 18.286 & 18.324 (0.054) \\
         \bottomrule
    \end{tabular}
    \caption{Empirical and theoretical expected stopping times under varying $\theta,p$, and $h$. Standard errors are reported in parentheses.}
    \label{tab:EmpiricalStopping_General}
\end{table}

The empirical results align closely with the theoretical predictions, as expected. Moreover, $p=0.5$ yields shorter stopping times than $p=0.15$,  since the problem is easier when $p=0.5$ due to the odds-ratio property (see Section \ref{subsec:CorrectDecision}). Similarly, larger values of $\theta$ reduce the stopping time, since they correspond to a looser IZ  (resulting in a smaller $H$ and larger $p-h$). 

\paragraph{Empirical stopping time under the worst-case scenario.}

Figure \ref{fig:EmpiricalStoppingTime_WorstCase} shows the empirical distributions of the stopping times for $p=h=0.15$ and $\theta\in \{1.2, 1.5\}$. The empirical distributions for $p=h=0.5$ are included in Appendix \ref{sec:AdditionalResults_BRFStoppingTime}.

\begin{figure}[h!]
\centering
\begin{subfigure}{.5\textwidth}
  \centering
  \includegraphics[width=\linewidth]{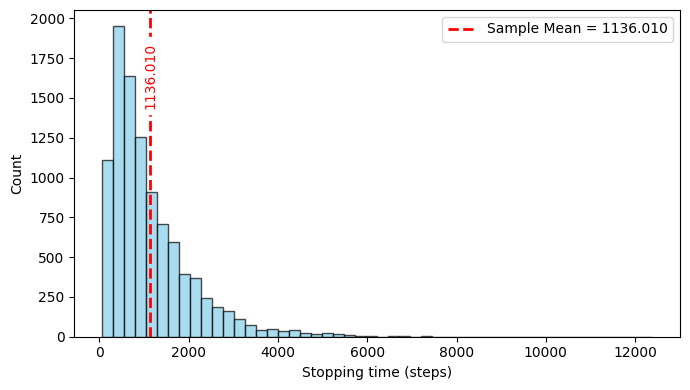}
  \caption{$\theta=1.2$}
  \label{fig:sub1}
\end{subfigure}%
\begin{subfigure}{.5\textwidth}
  \centering
  \includegraphics[width=\linewidth]{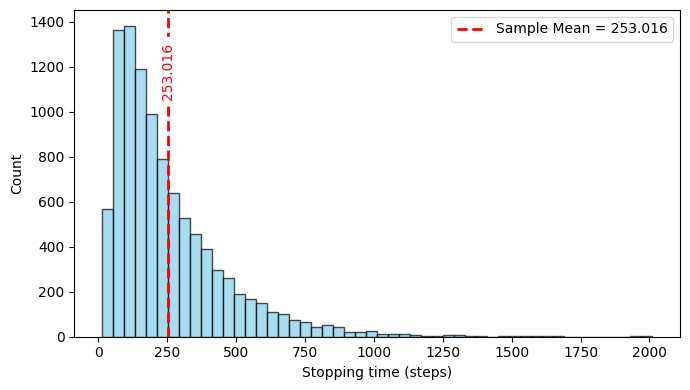}
  \caption{$\theta=1.5$}
  \label{fig:sub2}
\end{subfigure}
\caption{Empirical distributions of stopping times, when $p=h=0.15$.}
\label{fig:EmpiricalStoppingTime_WorstCase}
\end{figure}

Figure \ref{fig:EmpiricalStoppingTime_WorstCase} and Appendix \ref{sec:AdditionalResults_BRFStoppingTime} show that the empirical distributions of the stopping times in the worst-case scenario are left-skewed and that the average stopping time is not excessively large. As with earlier results, $p$ near $0.5$ or larger $\theta$ generally lead to shorter stopping times. 

\subsection{Comparison Between $\mathcal{BRF}$ and $\mathcal{RF}$} 
\label{subsec:RFComparison}

In this section, we compare the performance of $\mathcal{BRF}$ with $\mathcal{RF}$. 
Since both procedures depend in a similar manner on the number of systems and constraints, we do not consider multiple systems or constraints in this section. We discuss the thresholds and tolerance level used by $\mathcal{RF}$ for the comparison with $\mathcal{BRF}$ in Section \ref{sec:Thresholds_CompareRF}.
We then provide a numerical comparison of $\mathcal{BRF}$ and $\mathcal{RF}$ as a function of the batch size employed by $\mathcal{RF}$ in Section \ref{sec:BatchSize_CompareRF}.  

\subsubsection{Thresholds and Tolerance Level for $\mathcal{RF}$}
\label{sec:Thresholds_CompareRF}

The $\mathcal{RF}$ procedure requires a single tolerance level $\epsilon_\ell$ on each constraint $\ell$ whose role resembles that of $\theta_\ell$ but for a constraint on normal data. We need to identify an appropriate value of $\epsilon_\ell$ to ensure a fair comparison between $\mathcal{BRF}$ and $\mathcal{RF}$. Recall from Section \ref{subsec:CorrectDecision} that system $i$ is acceptable with respect to  $h_{\ell,m}^{(1)}$, where $m=1,\ldots,d_\ell^{(1)}$, if it has a probability $p_{i\ell}$ such that 
\begin{equation*}
    \frac{(1-p_{i\ell})h_{\ell,m}^{(1)}}{p_{i\ell} (1-h_{\ell,m}^{(1)})}<\theta_\ell \quad \text{ and } \quad \frac{p_{i\ell}(1-h_{\ell,m}^{(1)})}{(1-p_{i\ell})h_{\ell,m}^{(1)}}<\theta_\ell, 
\end{equation*}
which is equivalent to 
\begin{equation*}
    {\rm LB}_{\ell, m} <p_{i\ell}<{\rm UB}_{\ell, m}, \text{ where } {\rm LB}_{\ell, m} = \frac{h_{\ell,m}^{(1)}}{h_{\ell,m}^{(1)}+\theta_\ell(1-h_{\ell,m}^{(1)})} \text{ and } {\rm UB}_{\ell, m} = \frac{\theta_\ell h_{\ell,m}^{(1)}}{h_{\ell,m}^{(1)}(\theta_\ell-1)+1}.
    \label{eqn:BeRF_AcceptableRegion}
\end{equation*}
Note that ${\rm LB}_{\ell, m}$ is the largest probability of a desirable system, while ${\rm UB}_{\ell, m}$ is the smallest probability of an unacceptable system.

Since $\mathcal{RF}$ requires a single tolerance level $\epsilon_\ell$ on each constraint $\ell$, a natural approach for choosing $\epsilon_\ell$ that is consistent with the odds-ratio IZs (${\rm LB}_{\ell,m}$, ${\rm UB}_{\ell,m}$) is to compute the distances $h_{\ell,m}^{(1)}-{\rm LB}_{\ell,m}$ and ${\rm UB}_{\ell,m}-h_{\ell,m}^{(1)}$ for each threshold $m=1,\ldots,d_\ell$, and then take the smallest one:
\begin{align}
    \epsilon_{\ell,m}&=\min\{ {\rm UB}_{\ell,m}-h_{\ell,m}^{(1)}, h_{\ell,m}^{(1)}-{\rm LB}_{\ell,m}\} 
        = \begin{cases}
            \frac{h_{\ell,m}^{(1)} (1-h_{\ell,m}^{(1)})(\theta_\ell-1)}{h_{\ell,m}^{(1)}+\theta_\ell (1-h_{\ell,m}^{(1)})} & \text{ when } h_{\ell,m}^{(1)}\leq 0.5,   \\
            \frac{h_{\ell,m}^{(1)} (1-h_{\ell,m}^{(1)}) (\theta_\ell-1)}{h_{\ell,m}^{(1)} (\theta_\ell-1)+1} & \text{ when } h_{\ell,m}^{(1)}>0.5;
        \end{cases}  
        \label{eqn:epsilon_lm2}  \\
        \epsilon_\ell &= \min_{m=1,2,\ldots, d_\ell} \epsilon_{\ell,m}.
        \label{eqn:epsilon_RF}
\end{align} 
However, the conversion \eqref{eqn:epsilon_lm2} and \eqref{eqn:epsilon_RF} results in a small tolerance level $\epsilon_\ell$, which makes the performance of $\mathcal{RF}$ conservative. 
Therefore, we also consider an alternative formulation. On a normal constraint $\ell$, system $i$ is acceptable with respect to threshold $\tilde{h}_{\ell,m}^{(1)}$ and tolerance level $\tilde{\epsilon}_{\ell,m}$ if   $\tilde{h}_{\ell,m}^{(1)}-\tilde{\epsilon}_{\ell,m}< p_{i\ell} < \tilde{h}_{\ell,m}^{(1)}+\tilde{\epsilon}_{\ell,m}$ \citep{zhou2022finding}. Thus, by setting $p_{i\ell}-\tilde{\epsilon}_{\ell,m}={\rm LB}_{\ell,m}$ and $p_{i\ell}+\tilde{\epsilon}_{\ell,m}={\rm UB}_{\ell,m}$, we obtain an adjusted threshold $\tilde{h}_{\ell,m}^{(1)}$ and tolerance levels $\tilde{\epsilon}_{\ell,m}$ and $\tilde{\epsilon}_\ell$ as 
\begin{align}
    \tilde{h}_{\ell,m}^{(1)} &= \frac{{\rm LB}_{\ell, m} + {\rm UB}_{\ell, m}}{2} = \frac{h_{\ell,m}^{(1)}}{2} \left(\frac{\theta_\ell^2+1-h_{\ell,m}^{(1)} \left( \theta_\ell-1 \right)^2}{\left[h_{\ell,m}^{(1)}+\theta_\ell (1-h_{\ell,m}^{(1)})\right]\left[h_{\ell,m}^{(1)} (\theta_\ell-1)+1\right]} \right); \label{eqn:h_tilde} \\ 
    \tilde{\epsilon}_{\ell,m} &= \frac{ {\rm UB}_{\ell,m} - {\rm LB}_{\ell, m} }{2} = \frac{ h_{\ell,m}^{(1)} (1-h_{\ell,m}^{(1)}) \left(\theta_\ell^2-1\right) }{2 \left[h_{\ell,m}^{(1)}+\theta_\ell (1-h_{\ell,m}^{(1)})\right]\left[h_{\ell,m}^{(1)} (\theta_\ell-1)+1\right] }; \label{eqn:epsilon_tilde}  \\
    \tilde{\epsilon}_\ell &= \min_{m=1,2,\ldots,d_\ell} \tilde{\epsilon}_{\ell, m}.  \label{eqn:epsilon}
\end{align}
Observe that since $\tilde{h}_{\ell,m}^{(1)}\not=h_{\ell,m}^{(1)}$, this tolerance level setting solves a different feasibility problem and hence will in general yield different feasibility decisions than the original problem. 

We provide numerical results that compare the settings $\epsilon_\ell$ and $\tilde{\epsilon}_\ell$ in Section \ref{sec:BatchSize_CompareRF} and a discussion about the relationship between $h_{\ell, m}^{(1)}, \tilde{h}_{\ell,m}^{(1)},\epsilon_{\ell, m}, \tilde{\epsilon}_{\ell,m}$, and $\epsilon_\ell,\tilde{\epsilon}_\ell$ in Appendix \ref{sec:RF_Tolerance}.
A description of the $\mathcal{RF}$ procedure for subjective probability constraints is given in Appendix~\ref{sec:RF}.

\subsubsection{Numerical Comparison}
\label{sec:BatchSize_CompareRF}

As $\mathcal{RF}$ requires the observations to be normally distributed, it needs to use batch means as base observations for the feasibility check when the data is Bernoulli distributed. 
Therefore, we investigate how the batch size $b$ of $\mathcal{RF}$ affects both PCD and OBS for Bernoulli distributed data.

We consider a single system, single constraint scenario with eight true probabilities $p_{11}\in\{0.01, 0.05, 0.1, 0.15, 0.2, 0.3, 0.4, 0.5\}$, $d_1=1$ threshold, and odds-ratio $\theta_1\in \{1.2,1.5\}$. 
For each probability $p_{11}$ and odds-ratio $\theta_1$, we use the function $f^L_1(p_{11}, \theta_1)$ shown in Equation \eqref{eqn:thresholds} to determine the threshold for $\mathcal{BRF}$. In particular, we define the ``SC threshold'' as $h_{1,1}^{(1)}=f_{1}^{L}(p_{11}, \theta_1)$, while the ``non-SC threshold'' is defined as $h_{1,1}^{(1)}=f_{1}^{L}(p_{11}, \theta_1)/2$ (as explained in Remark \ref{remark:SlippageConfig}, the functions $f^L_\ell(p_{i\ell}, \theta_\ell)$ and $f^U_\ell(p_{i\ell}, \theta_\ell)$ provide the two ``most difficult'' thresholds for an odds-ratio $\theta_\ell>1$ on constraint $\ell$ with probability $p_{i\ell}$).
Consequently, the system is infeasible to the chosen threshold $h_{1,1}^{(1)}$. In addition, following the conversion from Equation \eqref{eqn:h_tilde}, the system is also infeasible to threshold $\tilde{h}_{1,1}^{(1)}$.
For $\mathcal{RF}$, we set the initial observation size $n_0=20$ and consider batch sizes $b\in\{1,2,4,8,16,32,64,100,200,300,400\}$ and the two tolerance levels settings $\epsilon_1$ and $\tilde{\epsilon}_1$ discussed in Section \ref{sec:Thresholds_CompareRF}. 
For each combination of $p_{11}, \theta_1$, and $b$, we estimate the PCD and the OBS for both the SC and non-SC threshold. 
In Table~\ref{tab:brf_rf_sc_nonsc}, the PCD Rate for each batch size reports how many of the eight probability values achieve a PCD above the nominal level; for example, ``2/8'' means that only two out of the eight probability settings satisfy the PCD being at least $1-\alpha$. Similarly, Table \ref{tab:brf_rf_sc_nonsc} shows the average OBS over the eight probability values. While Table \ref{tab:brf_rf_sc_nonsc} presents the results for $\theta_1=1.2$, we include the results for $\theta_1=1.5$ in Appendix \ref{sec:RFComparison_Additional}.
\begin{table}[h!]
\centering
\scriptsize
{\normalsize
\caption{PCD Rate and Average OBS for $\mathcal{BRF}$ and $\mathcal{RF}$ for $\theta_1=1.2$.}
\label{tab:brf_rf_sc_nonsc}
\renewcommand{\arraystretch}{0.8}
\begin{tabular}{lc||cc|cc|cc|cc}
\toprule
& &
\multicolumn{4}{c|}{SC Thresholds} &
\multicolumn{4}{c}{non-SC Thresholds} \\ \cline{3-10}
& &
\multicolumn{2}{c}{$\epsilon_1$} &
\multicolumn{2}{c|}{$\tilde{\epsilon}_1$} & \multicolumn{2}{c}{$\epsilon_1$} &
\multicolumn{2}{c}{$\tilde{\epsilon}_1$} \\ \cline{3-10}
& & PCD & Avg & PCD & Avg & PCD & Avg & PCD & Avg \\
& $b$ & Rate & OBS & Rate & OBS & Rate & OBS & Rate & OBS \\\midrule

$\mathcal{BRF}$ &  -- & 8/8 & 1850.3 & 8/8 & 1850.3 & 8/8 & 548.1 & 8/8 & 548.1 \\ \hline
\multirow{11}{*}{$\mathcal{RF}$}
& 1 & 3/8 & 1511.5 & 2/8 & 1425.3 & 5/8 & 1107.1 & 5/8 & 1021.5 \\
& 2 & 3/8 & 1474.7 & 2/8 & 1428.7 & 6/8 & 1102.1 & 6/8 & 1010.8 \\
& 4 & 4/8 & 1508.9 & 2/8 & 1404.5 & 7/8 & 1109.9 & 7/8 & 999.1 \\
& 8 & 6/8 & 1520.7 & 3/8 & 1408.2 & 7/8 & 1121.4 & 7/8 & 1025.2 \\
& 16 & 7/8 & 1553.4 & 4/8 & 1445.0 & 8/8 & 1199.1 & 8/8 & 1118.5 \\
& 32 & 7/8 & 1727.4 & 6/8 & 1619.6 & 8/8 & 1400.0 & 8/8 & 1325.7 \\
& 64 & 7/8 & 2183.6 & 6/8 & 2098.9 & 8/8 & 1904.7 & 8/8 & 1837.8 \\
& 100 & 8/8 & 2785.8 & 7/8 & 2690.7 & 8/8 & 2523.6 & 8/8 & 2458.4 \\
& 200 & 8/8 & 4539.4 & 7/8 & 4469.0 & 8/8 & 4289.9 & 8/8 & 4230.6 \\
& 300 & 8/8 & 6357.3 & 7/8 & 6299.8 & 8/8 & 6116.9 & 8/8 & 6078.7 \\
& 400 & 8/8 & 8217.4 & 8/8 & 8176.4 & 8/8 & 8035.5 & 8/8 & 8019.5 \\ 
\bottomrule
\end{tabular}}
\end{table}

As shown in Table~\ref{tab:brf_rf_sc_nonsc} and Appendix \ref{sec:RFComparison_Additional}, $\mathcal{BRF}$ attains the desired PCD level for all probability values under both SC and non-SC thresholds, while $\mathcal{RF}$ is sensitive to the choice of tolerance level and batch size $b$.
For $\theta_1=1.2$, $\mathcal{RF}$ requires a batch size as large as $b=400$ to achieve $8/8$ PCD Rate. 
Thus, applying $\mathcal{RF}$ to Bernoulli data may incur a non-negligible risk of incorrect decisions, whereas $\mathcal{BRF}$ consistently maintains the target confidence level.
When we require $\mathcal{RF}$ to achieve a PCD Rate comparable to that of $\mathcal{BRF}$, the resulting OBS is roughly 1.5 -- 9.3 times larger than that of $\mathcal{BRF}$.
The main difficulty for $\mathcal{RF}$ occurs when the true probability is small (e.g., $p_{11}=0.01$): the Bernoulli observations are then highly skewed, and a large batch size is needed for the normal approximation underlying $\mathcal{RF}$ to be accurate. 

In addition, Table \ref{tab:brf_rf_sc_nonsc} and Appendix \ref{sec:RFComparison_Additional} show that $\mathcal{RF}$ and $\mathcal{BRF}$ require more observations when $\theta_1=1.2$ than when $\theta_1=1.5$. This is expected as a larger $\theta_1$ indicates that the decision maker is more indifferent to feasibility decisions near the threshold, allowing decisions to be reached earlier. We also see that $\tilde{\epsilon}_1$ yields smaller Avg OBS than $\epsilon_1$. This is because $\tilde{\epsilon}_1$ is less conservative, as discussed in Section \ref{sec:Thresholds_CompareRF} and Appendix \ref{sec:RF_Tolerance}. Moreover, SC thresholds require more observations than non-SC thresholds for a fixed batch size $b$, as SC thresholds correspond to more difficult settings. 
Finally, when $b$ is large, the Avg OBS becomes similar across $\epsilon_1$ and $\tilde{\epsilon}_1$, SC and non-SC thresholds, and $\theta_1=1.2$ and $\theta_1=1.5$. 
This is due to the fact that $\mathcal{RF}$ requires $n_0$ initial sample batches for variance estimation; when the batch size is large, the feasibility decision is typically reached immediately after collecting that initial set of samples.

\subsection{Results for $\mathcal{MPB}$} 
\label{subsec:MBeRFExp}

In this section, we perform a numerical evaluation of the multi-pass procedure $\mathcal{MPB}$ compared with $\mathcal{BRF}$ that includes all thresholds in a single-pass.
Given that we propose three heuristic procedures, $\mathcal{BRF}_B^{(w)}$, $\mathcal{BRF}_N^{(w)}$,and $\mathcal{BRF}_{BN}^{(w)}$ for $w\geq 2$, we use $\mathcal{MPB}_B$, $\mathcal{MPB}_N$, and $\mathcal{MPB}_{BN}$ to denote the multi-pass procedures that incorporate $\mathcal{BRF}_B^{(w)}$, $\mathcal{BRF}_N^{(w)}$, and $\mathcal{BRF}_{BN}^{(w)}$, respectively. 

When thresholds are added sequentially, as in $\mathcal{MPB}_B$, $\mathcal{MPB}_N$, and $\mathcal{MPB}_{BN}$ when $w\geq 2$, we define the probability of correct decision $\widehat{{\rm PCD}}$ as 
\begin{align*}
    \widehat{{\rm PCD}} &= 
    \begin{cases}
        \ {\rm PCD}, & \text{ if } w= 1, \\
        \Pr \left( \cap_{i=1}^k \cap_{\ell=1}^s \cap_{h\in \cup_{u=1}^w T_\ell^{(u)} } {\rm CD}_{i\ell} (h) \right), & \text{ if } w>1.
    \end{cases}
\end{align*}
In Section \ref{subsubsec:HeuristicValidity}, we show that the three heuristic procedures $\mathcal{MPB}_B$, $\mathcal{MPB}_N$, and $\mathcal{MPB}_{BN}$ do not violate statistical validity in the numerical examples we consider. In Section \ref{subsubsec:HeuristicEfficiency}, we demonstrate that $\mathcal{MPB}_B$, $\mathcal{MPB}_N$, and $\mathcal{MPB}_{BN}$ can be substantially more efficient than $\mathcal{BRF}$.

\subsubsection{Validity of $\mathcal{MPB}_B$, $\mathcal{MPB}_N$, and $\mathcal{MPB}_{BN}$}
\label{subsubsec:HeuristicValidity}

In this section, we evaluate the validity of the three heuristic procedures $\mathcal{MPB}_B$, $\mathcal{MPB}_N$, and $\mathcal{MPB}_{BN}$ for a single system with two constraints. 
We set $p_{1\ell}=0.15$ and $\theta_\ell=1.5$ for $\ell=1,2$, and use $d_\ell=4$ thresholds for both $\ell=1,2$. In this example, we test feasibility through two passes (i.e., $w=2$) and consider the following two threshold configurations: 
\begin{itemize}
    \item Configuration 1: $T_{\ell}^{(1)} = \{f_\ell^L(p_{1\ell},1.5\theta_\ell), f_\ell^U(p_{1\ell},1.5\theta_\ell)\}$ and $T_{\ell}^{(2)} = \{f_\ell^L(p_{1\ell},\theta_\ell), f_\ell^U(p_{1\ell},\theta_\ell)\}$ for $\ell=1,2$.
    \item Configuration 2: $T_{\ell}^{(1)} = \left\{f_\ell^L(p_{1\ell},\theta_\ell), f_\ell^U(p_{1\ell},\theta_\ell)\right\}$ and $T_{\ell}^{(2)} = \left\{f_\ell^L(p_{1\ell},1.5\theta_\ell), f_\ell^U(p_{1\ell},1.5\theta_\ell)\right\}$ for $\ell=1,2$.
\end{itemize}

From Remark \ref{remark:SlippageConfig}, recall that $f_\ell^L (p_{1\ell}, \theta_\ell)$ and $f_\ell^U (p_{1\ell}, \theta_\ell)$ are the two ``most difficult'' thresholds for $p_{1\ell}$ when $\theta_\ell>0$. Further, $f_\ell^L (p_{1\ell}, \theta_\ell)$ decreases and $f_\ell^U (p_{1\ell}, \theta_\ell)$ increases as $\theta_\ell$ increases, and $f_\ell^L (p_{1\ell}, \theta_\ell)=f_\ell^U (p_{1\ell}, \theta_\ell)=p_{1\ell}$ when $\theta_\ell=1$. 
It follows that the functions $f_\ell^L (p_{1\ell}, 1.5\theta_\ell)$ and $f_\ell^U (p_{1\ell}, 1.5\theta_\ell)$ output thresholds that are easier for feasibility decisions, as they represent the most difficult thresholds associated with a larger odds ratio.
Thus, Configuration 1 tests easier thresholds first and then tests more difficult thresholds, whereas Configuration 2 is the opposite. We also consider $\mathcal{BRF}$ with the same four thresholds on the two constraints. 
We present the results in Table \ref{tab:berfwsingle1}, where we use ${\rm OBS}^{(u)}$ to denote the average required number of observations from the $u$th pass and $u=1,2$ for $\mathcal{MPB}_B$, $\mathcal{MPB}_N$, and $\mathcal{MPB}_{BN}$. 

\begin{table}[h!] 
\centering
\renewcommand{\arraystretch}{0.9}
\caption{Experimental results for $\mathcal{MPB}_B, \mathcal{MPB}_N, \mathcal{MPB}_{BN}$, and $\mathcal{BRF}$.}
\label{tab:berfwsingle1}
{\normalsize
\begin{tabular}{c | c c c c | c c c c}
\toprule
 & \multicolumn{4}{c}{Configuration 1} & \multicolumn{4}{c}{Configuration 2} \\ \cline{2-9}
 & \makecell{$\mathcal{MPB}_B$} & \makecell{$\mathcal{MPB}_N$} & \makecell{$\mathcal{MPB}_{BN}$} & $\mathcal{BRF}$ & \makecell{$\mathcal{MPB}_B$} & \makecell{$\mathcal{MPB}_N$} & \makecell{$\mathcal{MPB}_{BN}$} & $\mathcal{BRF}$ \\ \midrule
OBS$^{(1)}$ & 181.366 & 181.366 & 181.366 & -- & 365.684 & 365.684 & 365.684 & -- \\ 
OBS$^{(2)}$ & 177.726 & 173.754 & 114.830 &  -- & 0.972 & 0.771 & 0.362 & -- \\ \hline
OBS & 359.092 & 355.120 & 296.196 & 366.567 & 366.656 & 366.455 & 366.047 & 366.567 \\ \hline
$\widehat{{\rm PCD}}$ & 0.984 & 1.000 & 0.986 & 0.979 & 0.979 & 0.979 & 0.979 & 0.979 \\ 
\bottomrule
\end{tabular}
}
\end{table}
As shown in Table \ref{tab:berfwsingle1}, $\mathcal{MPB}_B$, $\mathcal{MPB}_N$, and $\mathcal{MPB}_{BN}$ all satisfy the $1 - \alpha = 0.95$ confidence level across both configurations. This indicates that while these procedures do not guarantee statistical validity, they can still meet the confidence level requirements of the decision maker. 

Focusing on the OBS, all $\mathcal{MPB}_B$, $\mathcal{MPB}_N$, and $\mathcal{MPB}_{BN}$ exhibit OBS levels comparable to $\mathcal{BRF}$ under Configuration 2, with $\mathcal{MPB}_{BN}$ slightly outperforming the others. Under Configuration 1, $\mathcal{MPB}_B$ and $\mathcal{MPB}_N$ show slightly less OBS than $\mathcal{BRF}$, whereas $\mathcal{MPB}_{BN}$ achieves approximately 20\% savings. The superior performance of $\mathcal{MPB}_{BN}$ relative to $\mathcal{MPB}_{B}$ and $\mathcal{MPB}_N$ is due to the efficiency of ${\cal BRF}_{BN}^{(2)}$ over ${\cal BRF}_{B}^{(2)}$ and ${\cal BRF}_{N}^{(2)}$, as discussed in Section \ref{subsec:HeuristicLaterPasses}.

As $\mathcal{MPB}_{BN}$ is guaranteed to perform superior compared with $\mathcal{MPB}_B$ and $\mathcal{MPB}_N$ in terms of the OBS, we focus on $\mathcal{MPB}_{BN}$ in the rest of the paper. 

\subsubsection{Efficiency in an Optimization Setting} 
\label{subsubsec:HeuristicEfficiency}

In this section, we demonstrate that $\mathcal{MPB}_{BN}$ can achieve substantial efficiency gains over $\mathcal{BRF}$ when the goal is to use subjective constraints to prune inferior systems. 
We consider a setting with 100 systems and two probability constraints. The true probabilities are as follows: for system 1, $p_{11} = p_{12} = 0.01$; for systems $i=2, 3, \ldots, 100$, $p_{i1} = p_{i2} = 0.5$. The objective is to identify the system with the lowest probabilities under both constraints (i.e., System 1).

To implement $\mathcal{MPB}_{BN}$, we set $w=2$ and assume that the decision maker begins with an initial threshold set $T_\ell^{(1)}=\{0.25\}$ for $\ell=1,2$. If only one system is deemed feasible, we return the feasible system and terminate the algorithm. Otherwise, if multiple feasible systems are identified, the decision maker refines the search by considering the threshold set $T_\ell^{(2)}=\{0.02,0.03, \ldots ,0.24\}$. Finally, if no systems are found feasible, the threshold set is expanded to $T_\ell^{(2)}=\{0.26, 0.27, \ldots, 0.49\}$. In contrast, under $\mathcal{BRF}$, the decision maker evaluates all thresholds in $\{0.02, 0.03, \ldots, 0.49\}$ simultaneously. 
The resulting comparison is reported in Table \ref{tab:largesavings}.

\begin{table}[h!] 
\centering
\setlength{\tabcolsep}{10pt}
\renewcommand{\arraystretch}{0.9}
\caption{$\widehat{{\rm PCD}}$ and OBS of $\mathcal{MPB}_{BN}$ and $\mathcal{BRF}$.}
\label{tab:largesavings}
{\normalsize
\begin{tabular}{ccc}
\toprule
    & $\mathcal{MPB}_{BN}$     & $\mathcal{BRF}$      \\ \midrule
    ${\rm OBS}^{(1)}$ & 10,256 & - \\  
    ${\rm OBS}^{(2)}$ & 0 & - \\  \hline
    OBS & 10,256 & 162,138 \\  \hline
    $\widehat{{\rm PCD}}$ & 0.999         & 0.984     \\ \bottomrule
\end{tabular}
}
\end{table}

Table \ref{tab:largesavings}  shows that $\mathcal{MPB}_{BN}$ achieves a $\widehat{{\rm PCD}}$ of 0.999, while $\mathcal{BRF}$'s PCD is 0.984. This difference arises because $\mathcal{MPB}_{BN}$ focuses on a single threshold of 0.25, allowing for an accurate feasibility check given that the $p_{i\ell}$ are far from 0.25 for all $i=1,\ldots,100$ and $\ell=1,2$. In contrast, $\mathcal{BRF}$ includes thresholds as close as 0.02 and 0.49 to the $p_{i\ell}$'s, leading to occasional failures in accurately checking feasibility. In addition, $\mathcal{MPB}_{BN}$ requires only about 6\% of OBS compared to $\mathcal{BRF}$. In conclusion, $\mathcal{MPB}_{BN}$ demonstrates more efficient performance compared to $\mathcal{BRF}$. This extreme case highlights $\mathcal{MPB}_{BN}$'s potential efficiency in practice.

\subsection{Inventory Problem}
\label{subsec:InventoryExp}

In this section, we demonstrate the performance of $\mathcal{BRF}$ and $\mathcal{MPB}_{BN}$ through the $(s, S)$ inventory problem discussed in Section \ref{sec:Intro}. We choose the setting as in \cite{law2007simulation}. 
The set of systems comprises all 77 possible combinations of $(s, S)$, where $s \in \{x: x = 20 + 2m, m = 0, 1, \ldots, 10\}$ and $S \in \{y: y = 40 + 10n, n = 0, 1, \ldots, 6\}$. The review period is one month and each replication runs for 12 review periods. The monthly demand follows a Poisson distribution with a mean of 25. The order cost is 3 per item, the fixed cost for ordering is 32 per order, the holding cost is 1 per item per review period, and the penalty cost is 5 per lost demand.

We consider two performance measures: (1) the probability of the total yearly cost exceeding 1400 (matching Section \ref{sec:Intro} with all costs measured in units of \$1,000), and (2) the probability of a stockout occurring within a year. We use $p_{i1}$ and $p_{i2}$ to denote the true value of (1) and (2), respectively, for $i=1,2,\ldots, 77$, and start by obtaining precise estimates of $p_{i1}$ and $p_{i2}$ based on simulation with 1,000,000 replications. 
Figure \ref{fig:InventoryConstr} in Appendix \ref{subsec:Experiments_Inventory} shows  the values of the estimated performance measures for the 77 systems considered as functions of $(s, S)$. Figure \ref{fig:TwoPerformance} shows a scatter plot of the estimated true probabilities $(p_{i1}, p_{i2})$ of the 77 systems considered. 
In this example, $p_{i1}$ ranges from 0.0035 to 0.9872, while $p_{i2}$ ranges from $4.9\times 10^{-5}$ to 0.9256.

\begin{figure}[h!]
    \centering
    \includegraphics[width=0.5\linewidth]{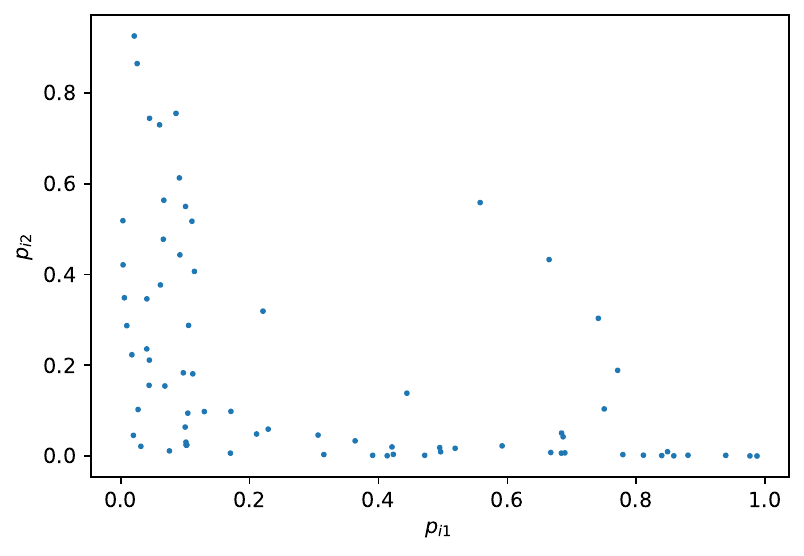}
    \caption{Estimated true probabilities $(p_{i1}, p_{i2})$ for the 77 systems.}
    \label{fig:TwoPerformance}
\end{figure}

The decision maker wishes to identify systems that satisfy both 
$p_{i1} \leq h_1$ and $p_{i2} \leq  h_2$, where $h_1$ and $h_2$ are chosen subjectively. We consider three procedures, $\mathcal{RF}, \mathcal{BRF}$, and $\mathcal{MPB}_{BN}$, and provide the standard error (s.e.) to quantify the variability of the estimates for OBS.  

In this setup, $\mathcal{BRF}$ and ${\cal RF}$ test all thresholds, $h_{\ell} \in \{0.01, 0.03, 0.05, \ldots, 0.49, 0.51\}$ for both $\ell=1,2$, whereas $\mathcal{MPB}_{BN}$ tests the thresholds $h_{\ell} \in T_\ell^{(1)}= \{0.11, 0.21, 0.31, 0.41, 0.51\}$ in the first pass to narrow down an appropriate threshold range, and in the second pass, it tests the thresholds $h_{\ell} \in T_\ell^{(2)}= \{0.01, 0.03, 0.05, 0.07, 0.09\}$ for both $\ell=1,2$ (because in all cases multiple systems are identified as feasible to threshold 0.11 for both constraints in the first pass).

We set $\theta_{1} = \theta_{2} = 1.5$, for $\mathcal{BRF}$ and $\mathcal{MPB}_{BN}$, and set $n_0=20$ and $b=32$ for $\mathcal{RF}$. The tolerance levels of $\mathcal{RF}$ are set as $\tilde{\epsilon}_\ell$, see Section \ref{sec:Thresholds_CompareRF}. Thus, the parameters of $\mathcal{RF}$ are not chosen conservatively. 
We do not incorporate CRN. The results are presented in Table~\ref{tab:multi2const}.  

\begin{table}[h!] 
\centering
\renewcommand{\arraystretch}{0.8}
\caption{$\widehat{{\rm PCD}}$ and OBS for the inventory problem leveraging $\mathcal{MPB}$ for pruning inferior systems.}
\label{tab:multi2const}
{\normalsize
\begin{tabular}{c|ccc}
\toprule
    & $\mathcal{MPB}_{BN}$     & $\mathcal{BRF}$ & $\mathcal{RF}$      \\ \midrule
$\widehat{{\rm PCD}}$ & 0.994        & 0.991    & 0.999 \\ \hline
${\rm OBS}^{(1)}$ & 79,526 & -- & -- \\ \hline
${\rm OBS}^{(2)}$ & 36,876 & -- & -- \\ \hline
OBS (s.e.) & 116,402 (401.33) & 352,507 (1183.53) & 4,707,590 (9713.13)\\ \bottomrule
\end{tabular}
}
\end{table}

From the perspective of $\widehat{{\rm PCD}}$, all three procedures successfully guarantee the desired statistical validity. However, in terms of OBS, $\mathcal{MPB}_{BN}$ requires only about 33\% of the observations compared to $\mathcal{BRF}$, demonstrating that $\mathcal{MPB}_{BN}$ is a more efficient procedure with this setting. 
In this setting, $\mathcal{RF}$ requires approximately 13 times more observations than $\mathcal{BRF}$, demonstrating that both $\mathcal{MPB}_{BN}$ and $\mathcal{BRF}$ are more efficient than $\mathcal{RF}$. This confirms the superior efficiency of $\mathcal{MPB}_{BN}$ and $\mathcal{BRF}$ over $\mathcal{RF}$ in terms of Bernoulli data.

Specifically, $\mathcal{MPB}_{BN}$ operates on relatively spread out thresholds (i.e., $\{0.11, 0.21, 0.31, 0.41,0.51\}$) during the first pass, which allowed it to check the feasibility for five thresholds with a small number of OBS. In the second pass, $\mathcal{MPB}_{BN}$ saves OBS compared to $\mathcal{BRF}$, as it avoids the unnecessary examination of all thresholds for all systems. This highlights $\mathcal{MPB}_{BN}$’s practical advantage in efficiently reducing the number of observations needed. 

Additional results that demonstrate the efficiency of $\mathcal{BRF}$ and $\mathcal{MPB}_{BN}$ when they are compared against $\mathcal{RF}$ in this realistic inventory example are provided in Appendix \ref{sec:ComparingBeRFandRF}. This includes results that test different values of $\theta_1,\theta_2, b$ for all thresholds (as opposed to only the thresholds needed to identify the best systems) and results on the number of systems deemed feasible under different combinations of thresholds (compared with the actual number of feasible systems based on the true system performance). We provide results for the use of CRN in Appendix \ref{subsec:Experiments_CRN}. Finally, we include results for when the decision maker is more sensitive to the stockout probability than the yearly cost in Appendix \ref{sec:SensitiveStockout_Results}, which demonstrate similar performance as shown in Table \ref{tab:multi2const}.

\section{Conclusion}
\label{sec:conclusion}

In this paper, we address the problem of identifying feasible systems among a finite number of simulated alternatives, when the constraints are on probabilities. Thus, the simulation observations are Bernoulli distributed and we further consider subjective probability constraints where thresholds vary. We develop novel procedures that employ a random walk model with odds-ratio IZ parameters and recycle observations for different thresholds for efficiency. Our $\mathcal{BRF}$ procedure is statistically valid and significantly outperforms the earlier $\mathcal{RF}$ procedure that requires an initial sample size for variance estimation and batching to achieve the approximate normality of basic observations. Our $\mathcal{MPB}$ procedure extends beyond $\mathcal{BRF}$ by checking only the necessary thresholds in a sequential manner. Through experimental results, $\mathcal{MPB}$ demonstrates its potential as a more efficient procedure compared even with $\mathcal{BRF}$ when the decision maker sequentially adds thresholds in multiple passes and uses them to prune inferior systems.

\section*{Acknowledgments}
The second author was supported by NSF under grants CMMI--2127778 and CMMI--2348409. The third author was supported by CMMI--2348409.

\singlespacing
{\footnotesize 
	\bibliographystyle{apalike}
	\bibliography{myref}
}

\clearpage
\appendix

\section*{Appendices}

Appendices \ref{sec:AdditionalResults_BRFStoppingTime}, \ref{sec:RFComparison_Additional}, and \ref{sec:Experiments_Additional} show additional numerical results for Sections \ref{subsec:BRFStoppingTime}, \ref{subsec:RFComparison}, and \ref{subsec:InventoryExp}, respectively. 
Appendix \ref{sec:RF_Tolerance} includes discussion about the tolerance levels and thresholds used in Procedure ${\cal RF}$ for the numerical experiments. Appendix \ref{sec:RF} provides detailed discussion on how a decision maker uses ${\cal RF}$ to perform feasibility checks for subjective probability constraints. 

\section{Additional Numerical Results for Section \ref{subsec:BRFStoppingTime}}
\label{sec:AdditionalResults_BRFStoppingTime}

Figure \ref{fig:EmpiricalStoppingTime_WorstCase_Additional} shows the empirical distributions of the stopping times for $p=0.5$ and $\theta\in \{1.2, 1.5\}$. 

\begin{figure}[h!]
\centering
\begin{subfigure}{.5\textwidth}
  \centering
  \includegraphics[width=\linewidth]{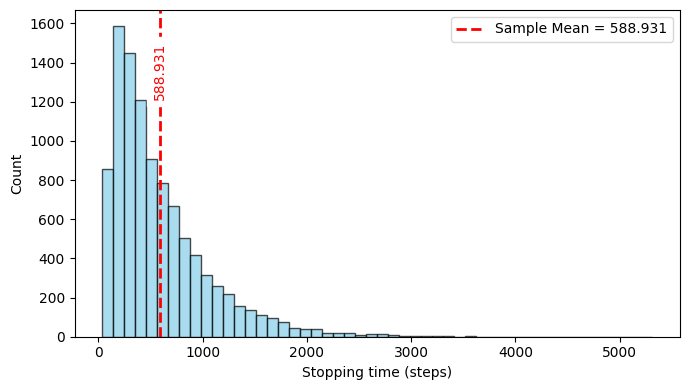}
  \caption{$\theta=1.2$}
  \label{fig:sub1}
\end{subfigure}%
\begin{subfigure}{.5\textwidth}
  \centering
  \includegraphics[width=\linewidth]{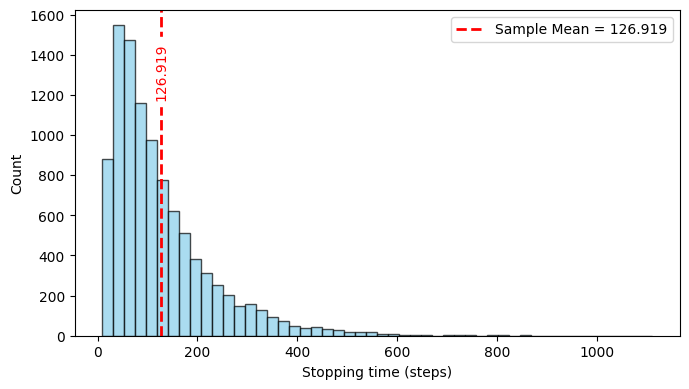}
  \caption{$\theta=1.5$}
  \label{fig:sub2}
\end{subfigure}
\caption{Empirical distributions of stopping times, when $p=h=0.5$.}
\label{fig:EmpiricalStoppingTime_WorstCase_Additional}
\end{figure}

\section{Discussion about the Tolerance Levels and Thresholds Used in $\mathcal{RF}$}
\label{sec:RF_Tolerance}

As our discussion focuses on a particular constraint $\ell$, we drop the subscript $\ell$ in $\epsilon_{\ell,m}, \theta_\ell, d_\ell, h_{\ell,m}^{(1)}$, and $\tilde{h}_{\ell,m}^{(1)}$ for simplicity. We also omit the superscript $(1)$ in $h_{\ell,m}^{(1)}$ and $\tilde{h}_{\ell,m}^{(1)}$.

\paragraph{Relation between $h_m$ and $\tilde{h}_m$.} 
According to \eqref{eqn:h_tilde}, with a given odds-ratio $\theta>1$, $\tilde{h}_m$ (as well as ${\rm LB}_m$ and ${\rm UB}_m$) depends on $h_m$ as follows:
\begin{align}
    \tilde{h}_m &= \frac{h_m}{2} \times \frac{\theta^2+1-h_m(\theta-1)^2}{[h_m+\theta (1-h_m)][h_m(\theta-1)+1]}  \label{eqn:AdjustedH}
\end{align}

Figure \ref{fig:RelationBetweenThresholds}, illustrates the relationship between $\tilde{h}_m$ and $h_m$, with $h_m$ ranging from 0 to 1. We consider two settings: a practically relevant case with $\theta=1.5$, and a more extreme case with $\theta=5$ included for illustrative purposes.
\begin{figure}[h!]
\centering
\begin{subfigure}{.5\textwidth}
  \centering
  \includegraphics[width=0.9\linewidth]{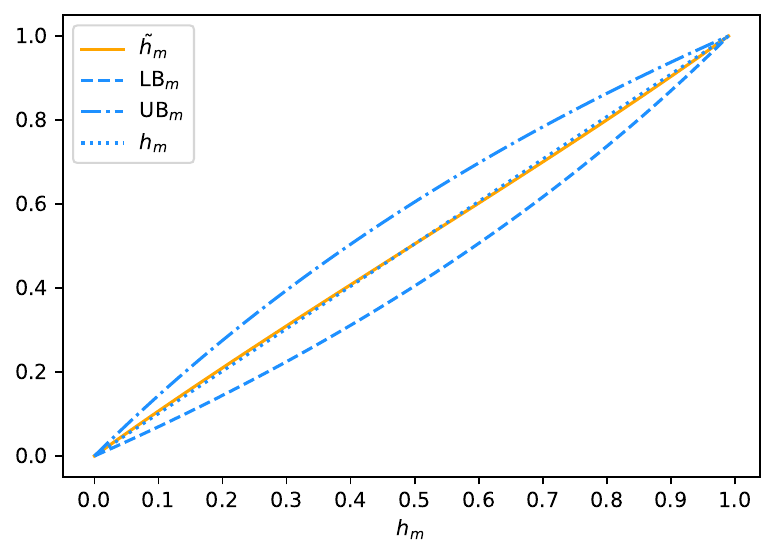}
  \caption{$\theta=1.5$}
  \label{fig:AdjustedThreshold_RF_FixedTheta1.5}
\end{subfigure}%
\begin{subfigure}{.5\textwidth}
  \centering
  \includegraphics[width=0.9\linewidth]{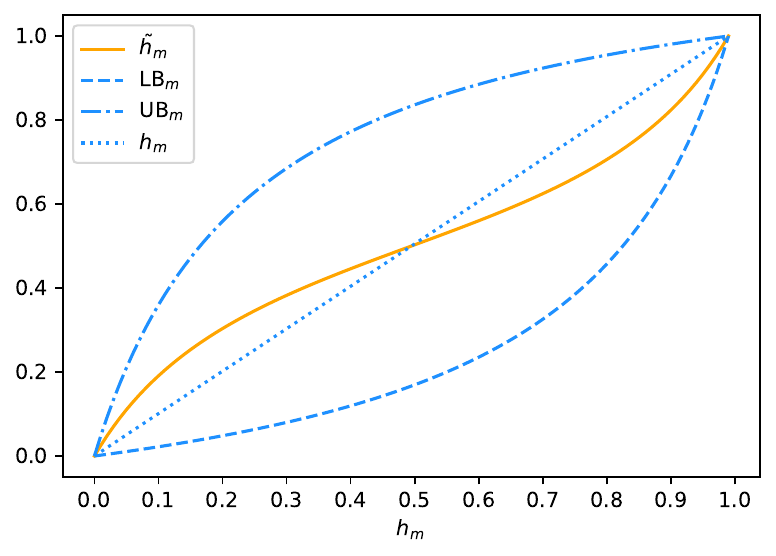}
  \caption{$\theta=5$}
  \label{fig:AdjustedThreshold_RF_FixedTheta5}
\end{subfigure}
\caption{Demonstration of $\tilde{h}_m, {\rm LB}_m$, and ${\rm UB}_m$ depending on $h_m$, where we consider $\theta\in \{1.5, 5\}$. }
\label{fig:RelationBetweenThresholds}
\end{figure}

We have the following key observations about the relationship between $h_m$ and $\tilde{h}_m$ that can be seen from Figure \ref{fig:RelationBetweenThresholds} and verified from Equation \eqref{eqn:AdjustedH}:
\begin{itemize}
    \item $\tilde{h}_m> h_m$ when $0< h_m< 0.5$ and $\tilde{h}_m< h_m$ when $0.5< h_m< 1$. Further, $\tilde{h}_m=h_m$ when $h_m=0, 0.5$, or 1. 
    \item For a fixed $h_m$, the difference between $\tilde{h}_m$ and $h_m$ is larger when $\theta$ is larger. 
\end{itemize}

\paragraph{Relation between $\epsilon$ and $\tilde{\epsilon}$.}
According to Equations \eqref{eqn:epsilon_lm2} and \eqref{eqn:epsilon_tilde}, the tolerance levels $\epsilon_m$ and $\tilde{\epsilon}_m$ associated with threshold $h_m$ and odds-ratio $\theta$ can be expressed by the functions
\begin{align*}
        E(h_m, \theta) &= \begin{cases}
            \frac{h_m (1-h_m)(\theta-1)}{h_m+\theta (1-h_m)} & \text{ when } h_m\leq 0.5   \\
            \frac{h_m (1-h_m) (\theta-1)}{h_m (\theta-1)+1} & \text{ when } h_m>0.5
        \end{cases}
    \quad \text{ and } \quad
    \tilde{E}(h_m, \theta) = \frac{h_m (1-h_m) (\theta^2-1) }{2 \left[ h_m+\theta (1-h_m) \right] \left[ h_m (\theta-1)+1 \right]  },
\end{align*} 
respectively, and the overall tolerance level for the constraint is taken as the minimum across all thresholds: $\epsilon = \min_{m=1,\ldots,d} \epsilon_m$ and $\tilde{\epsilon}=\min_{m=1,\ldots,d} \tilde{\epsilon}_m$. Figure \ref{fig:ToleranceLevelFunctions} shows a demonstration of the functions $E(h_m,\theta)$ and $\tilde{E}(h_m, \theta)$ where we fix $\theta=1.5$. 
\begin{figure}[h!]
\centering
\includegraphics[width=0.5\linewidth]{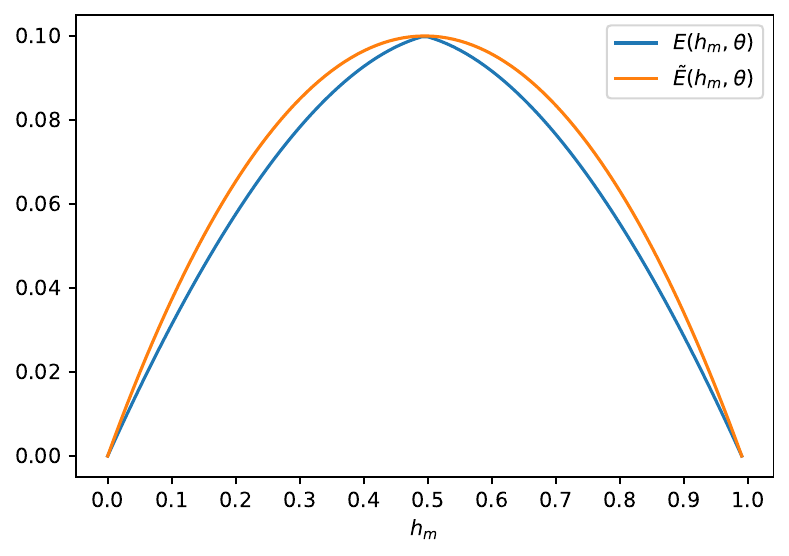}
\caption{Demonstration of the functions $E(h_m, \theta)$ and $\tilde{E}(h_m, \theta)$ when $\theta=1.5$.}
\label{fig:ToleranceLevelFunctions}
\end{figure}

There are four key observations about the relationship between $\epsilon_m,\tilde{\epsilon}_m$, and $h_m$:
\begin{itemize}
    \item For a fixed $\theta$, both $E(h_m, \theta)$ and $\tilde{E}(h_m, \theta)$ approach 0 when the threshold $h_m$ approaches 0 or 1. This aligns with our discussion in Section \ref{subsec:CorrectDecision} that feasibility checks become more challenging when thresholds are extreme, while they are easiest when $h_m$ is near 0.5. In fact, both $E(h_m, \theta)$ and $\tilde{E}(h_m, \theta)$ attain their maximum at $h_m=0.5$, where $E(0.5, \theta)=\tilde{E}(0.5, \theta)=\frac{\theta-1}{2(\theta+1)}$.

    \item $\tilde{E}(h_m, \theta)\geq E(h_m, \theta)$ for all combinations of $h_m$ and $\theta$. This agrees with our discussion in Section \ref{sec:Thresholds_CompareRF} that $E(h_m, \theta)$ provides a more conservative setting compared with $\tilde{E}(h_m, \theta)$. 
    \item Both $E(h_m, \theta)$ and $\tilde{E}(h_m, \theta)$ are symmetric about $h_m = 0.5$; that is, $E(h_m, \theta)= E(1-h_m, \theta)$ and $\tilde{E}(h_m, \theta)= \tilde{E}(1-h_m, \theta)$. 
    \item Because the constraint’s tolerance level is defined as the minimum across all $\epsilon_m$ or $\tilde{\epsilon}_m$, the presence of any threshold $h_m$ close to 0 or 1 forces $\epsilon$ or $\tilde{\epsilon}$ to be small. Conversely, if all thresholds are near 0.5, the resulting $\epsilon$ or $\tilde{\epsilon}$ will not be excessively small. A small tolerance level increases the number of observations required by $\mathcal{RF}$. This is a benefit of our odds-ratio approach, see Section \ref{subsec:CorrectDecision}.
\end{itemize}

\section{$\mathcal{RF}$ for Subjective Probability Constraints}
\label{sec:RF}

In this section, we provide a detailed discussion on how the decision maker can use $\mathcal{RF}$, due to \citet{zhou2022finding}, to perform feasibility checks for subjective probability constraints.

To use Procedure $\mathcal{RF}$ to solve the feasibility problem involving constraints on probabilities, we need two modifications: (1) treating batch means of the Bernoulli observation $Y_{i\ell n}$ as basic observations, and (2) determining the tolerance level that is used by $\mathcal{RF}$ for fair comparison with $\mathcal{BRF}$. A detailed discussion of choosing the equivalent tolerance level for $\mathcal{RF}$ is provided in Section \ref{subsec:RFComparison} and Appendix \ref{sec:RF_Tolerance}. We now discuss treating batch means as basic observations. 

The $\mathcal{RF}$ procedure is designed for normally distributed observations. Since batch means of Bernoulli observations can be approximated as normally distributed data, 
we use batch means of $Y_{i\ell n}$ as basic observations. 
Specifically, we let $b$ be the batch size and let ${B}_{i\ell n}$ be the $n$th basic observation (i.e., the $n$th batch mean), where 
\begin{align}
    {B}_{i\ell n} = \frac{1}{b}\sum_{n'=b(n-1)+1}^{bn} Y_{i\ell {n'}}.  \label{eqn:BatchMeans_RF}
\end{align}
In other words, to generate the $n$th basic observation ${B}_{i\ell n}$ for $\mathcal{RF}$, one needs to obtain $b$ Bernoulli observations $Y_{i\ell n'}$ where $n'=b(n-1)+1, \ldots, bn$. We let $\bar{B}_{i\ell}{(r)}=\frac{1}{r} \sum_{n=1}^r B_{i\ell n}$ be the average of $r$ batch means with a batch size of $b$.

To implement the $\mathcal{RF}$ procedure, we need additional notation as follows.
\begin{align}
n_0 &\equiv \text { the initial observation size for each system } \left(n_0 \geq 2\right); \nonumber \\
S_{i \ell}^2\left(n_0\right) &\equiv \text { the sample variance of } {B}_{i \ell 1}, \ldots, {B}_{i \ell n_0}; \nonumber \\
R\left(r_i ; v, w, z\right) &\equiv \max \left\{0, \frac{\left(n_0-1\right) w z}{v}-\frac{v}{2 c} r_i\right\} \text { for } v, w, z \in \mathbb{R}^{+} \text {and } c \in\{1,2, \ldots, \infty\};  \nonumber \\
g(\eta) &\equiv 
\begin{cases}
    \sum_{j=1}^c(-1)^{j+1}\left(1-\frac{1}{2} \mathbb{I}(j=c)\right) \times\left(1+\frac{2 \eta(2 c-j) j}{c}\right)^{-\left(n_0-1\right) / 2}, & c \in \mathbb{N}^{+},  \\
    \int_0^{\infty} \frac{1}{1+\exp (2 \eta x)} \times \frac{1}{2^{\left(n_0-1\right) / 2} \Gamma\left(\frac{n_0-1}{2}\right)} x^{\frac{n_0-1}{2} -1} e^{-x / 2} dx, & c=\infty, \label{eqn:beta_rf}
\end{cases}  \nonumber
\end{align}
where $\Gamma(\cdot)$ is the gamma function. For the $g(\eta)$ function, the choice of $c=1$ is the most popular \citep{zhou2022finding}, which is also used in our experiments. Algorithm \ref{alg:rf} provides a detailed description of the $\mathcal{RF}$ procedure in the presence of subjective probability constraints using the less conservative tolerance levels $\tilde{\epsilon}_1, \ldots, \tilde{\epsilon}_s$ from Equation \eqref{eqn:epsilon} (to employ the tolerance levels $\epsilon_1, \ldots, \epsilon_s$ from Equation \eqref{eqn:epsilon_RF}, replace $\tilde{\epsilon}_\ell$ and $\tilde{h}_{\ell,m}^{(1)}$ in Algorithm \ref{alg:rf} by $\epsilon_\ell$ and $h_{\ell,m}^{(1)}$, respectively, for $\ell=1,\ldots,s$ and $m=1,\ldots,d_\ell$).  

\begin{algorithm}[h!]
	\caption{Procedure $\mathcal{RF}$ for subjective probability constraints}\label{alg:rf}
{\fontsize{10}{14}\selectfont	
\begin{algorithmic}
	\State [{\bf Setup}:]  Choose confidence level $1-\alpha$, initial observation size $n_0$, threshold set $\{h_{\ell,1}^{(1)}, h_{\ell,2}^{(1)}, \ldots, h_{\ell, d_\ell}^{(1)}\}$, odds-ratio IZ parameters $\theta_\ell$ for $\ell=1,2,\ldots,s$, and the batch size $b$. Set $\Gamma =\{1,2,\ldots,k\}$. 
    Calculate the thresholds $\{\tilde{h}_{\ell,1}^{(1)},\tilde{h}_{\ell,2}^{(1)},\ldots,\tilde{h}_{\ell,d_\ell}^{(1)}\}$ and tolerance level $\tilde{\epsilon}_\ell$ from Equations \eqref{eqn:h_tilde} and \eqref{eqn:epsilon} for $\ell=1,\ldots,s$. Set $\eta_\ell$ such that $g(\eta_\ell) = \beta_\ell$
    where $\beta_\ell$ is determined as in Equation \eqref{eqn:betaell}.

	\For{each system $i \in \Gamma$} 
        \State [{\bf Initialization}:] 
            \setlength{\itemindent}{0.26in}
            \begin{itemize}
            \item Obtain $b n_0$ observations $Y_{i\ell1}, \ldots, Y_{i\ell (b n_0)}$.
            \item Calculate
            $B_{i \ell n}$ as in Equation \eqref{eqn:BatchMeans_RF} for $n=1,\ldots,n_0$. 
            \item Compute the sample mean and sample variance of the batch means $B_{i\ell 1}, \ldots, B_{i\ell n_0}$ as $\bar{B}_{i \ell}(n_0)$ and $S_{i \ell}^2\left(n_0\right)$, respectively.  
            \item Set $r_{i}=n_0, {\rm ON} = \{1,2,\ldots, s\}$, and ${\rm ON}_\ell = \{1,2,\ldots, d_\ell\}$ for $\ell = 1, 2, \ldots, s$.
            \end{itemize}
        \State [{\bf Feasibility Check}:]  
        \For{each constraint $\ell \in {\rm ON}$}
        \For{each threshold $m \in {\rm ON}_\ell$}
		\begin{itemize}
			\setlength{\itemindent}{0.3in}
            \item[] If $\bar{B}_{i \ell}\left(r_i\right)+R(r_i ; \tilde{\epsilon}_{\ell}, \eta_\ell, S_{i \ell}^2\left(n_0\right))/{r_i} \leq \tilde{h}_{\ell,m}^{(1)}$, set $ Z_{i\ell}^{(m)}=1$ and ${\rm ON}_\ell = {\rm ON}_\ell \setminus \{m\}$;
			\item[] Else if $\bar{B}_{i \ell}\left(r_i\right)-R(r_i ; \tilde{\epsilon}_{\ell}, \eta_\ell, S_{i \ell}^2\left(n_0\right))/{r_i} \geq \tilde{h}_{\ell,m}^{(1)}$, set $ Z_{i\ell}^{(m)}=0$ and ${\rm ON}_\ell = {\rm ON}_\ell \setminus \{m\}$. 
		\end{itemize}
        \EndFor
        \State If ${\rm ON}_\ell = \emptyset$, set ${\rm ON} = {\rm ON}\setminus \{\ell\}$
        \EndFor
    \State [{\bf Stopping Condition}:] 
		 \begin{itemize}
			\item[] If ${\rm ON}=\emptyset$, return $Z_{i\ell}^{(m)}$ for $\ell=1,\ldots,s$ and $m=1,\ldots,d_\ell$. Otherwise, take $b$ additional observation $Y_{i\ell (r_{i}b+1)}, \ldots, Y_{i \ell (b r_i+b)}$ for all $\ell=1,\ldots,s$, and set $r_{i}=r_{i}+1$. Update $\bar{B}_{i \ell}\left(r_i\right)$ for $\ell \in {\rm ON}$ and go to [{\bf Feasibility Check}].
		\end{itemize} 
    \EndFor

\end{algorithmic}
}
\end{algorithm}

\section{Additional Numerical Results for Section \ref{subsec:RFComparison}}
\label{sec:RFComparison_Additional}

Table \ref{tab:brf_rf_sc_nonsc_additional} shows experimental results comparing $\epsilon_1$ and $\tilde{\epsilon}_1$ for $\theta_1=1.5$ under both SC or non-SC thresholds. 

\begin{table}[h!]
\centering
\scriptsize
{\normalsize
\caption{PCD Rate and Average OBS for $\mathcal{BRF}$ and $\mathcal{RF}$ for $\theta_1=1.5$.}
\label{tab:brf_rf_sc_nonsc_additional}
\renewcommand{\arraystretch}{0.8}
\begin{tabular}{lc||cc|cc|cc|cc}
\toprule
& &
\multicolumn{4}{c|}{SC Thresholds} &
\multicolumn{4}{c}{non-SC Thresholds} \\ \cline{3-10}
& & 
\multicolumn{2}{c}{$\epsilon_1$} &
\multicolumn{2}{c|}{$\tilde{\epsilon}_1$} & \multicolumn{2}{c}{$\epsilon_1$} &
\multicolumn{2}{c}{$\tilde{\epsilon}_1$} \\ \cline{3-10}
& & PCD & Avg & PCD & Avg & PCD & Avg & PCD & Avg \\
& $b$ & Rate & OBS & Rate & OBS & Rate & OBS & Rate & OBS \\\midrule

$\mathcal{BRF}$ & --  & 8/8 & 432.4 & 8/8 & 432.4 & 8/8 & 226.4 & 8/8 & 226.4 \\ \cline{2-10}

\multirow{11}{*}{$\mathcal{RF}$}
& 1 & 3/8 & 486.7 & 1/8 & 418.0 & 5/8 & 581.3 & 5/8 & 477.1 \\
& 2 & 5/8 & 493.3 & 2/8 & 416.2 & 6/8 & 585.7 & 6/8 & 478.6 \\
& 4 & 6/8 & 498.6 & 3/8 & 429.0 & 7/8 & 589.9 & 7/8 & 491.9 \\
& 8 & 7/8 & 539.0 & 5/8 & 471.8 & 7/8 & 634.0 & 7/8 & 532.7 \\
& 16 & 7/8 & 646.0 & 6/8 & 573.2 & 8/8 & 730.5 & 8/8 & 643.8 \\
& 32 & 7/8 & 900.6 & 7/8 & 843.0 & 8/8 & 984.0 & 8/8 & 901.6 \\
& 64 & 8/8 & 1463.2 & 7/8 & 1409.1 & 8/8 & 1534.2 & 8/8 & 1461.5 \\
& 100 & 8/8 & 2110.2 & 7/8 & 2068.1 & 8/8 & 2168.5 & 8/8 & 2007.0 \\
& 200 & 8/8 & 4016.0 & 8/8 & 4007.1 & 8/8 & 4025.7 & 8/8 & 4007.1 \\
& 300 & 8/8 & 6001.2 & 8/8 & 6000.5 & 8/8 & 6001.4 & 8/8 & 6000.1 \\
& 400 & 8/8 & 8000.1 & 8/8 & 8000.0 & 8/8 & 8000.1 & 8/8 & 8000.0 \\ 
\bottomrule
\end{tabular}}
\end{table}

\section{Additional Numerical Results for Section \ref{subsec:InventoryExp}}
\label{sec:Experiments_Additional}

In this section, we provide additional discussion related to Section \ref{subsec:InventoryExp}. Section \ref{subsec:Experiments_Inventory} includes additional results regarding the two performance measures tested in Section \ref{subsec:InventoryExp}. Section \ref{sec:ComparingBeRFandRF} provides results that demonstrate the efficiency of ${\cal BRF}$ and ${\cal MPB}_{BN}$ compared with ${\cal RF}$. Section \ref{subsec:Experiments_CRN} addresses the use of CRN. Section \ref{sec:SensitiveStockout_Results} provides additional results for the scenario when the decision maker is more sensitive to one of the performance measures compared with the other. 

\subsection{Additional Results for the Performance Measures}
\label{subsec:Experiments_Inventory}

In this section, we present the values of the two performance measures considered in the inventory example from Section \ref{subsec:InventoryExp}.
Figure \ref{fig:InventoryConstr} shows the plot of values of the two performance measures as functions of $s$ and $S$, where $s\in \{x: x=20+20m, m=0,1,\ldots,10\}$ and $S\in \{y:y=40+10n, n=0,1,\ldots,6\}$. 

\begin{figure}[h!]
\centering
\begin{subfigure}{.45\textwidth}
  \centering
  \includegraphics[width=\linewidth]{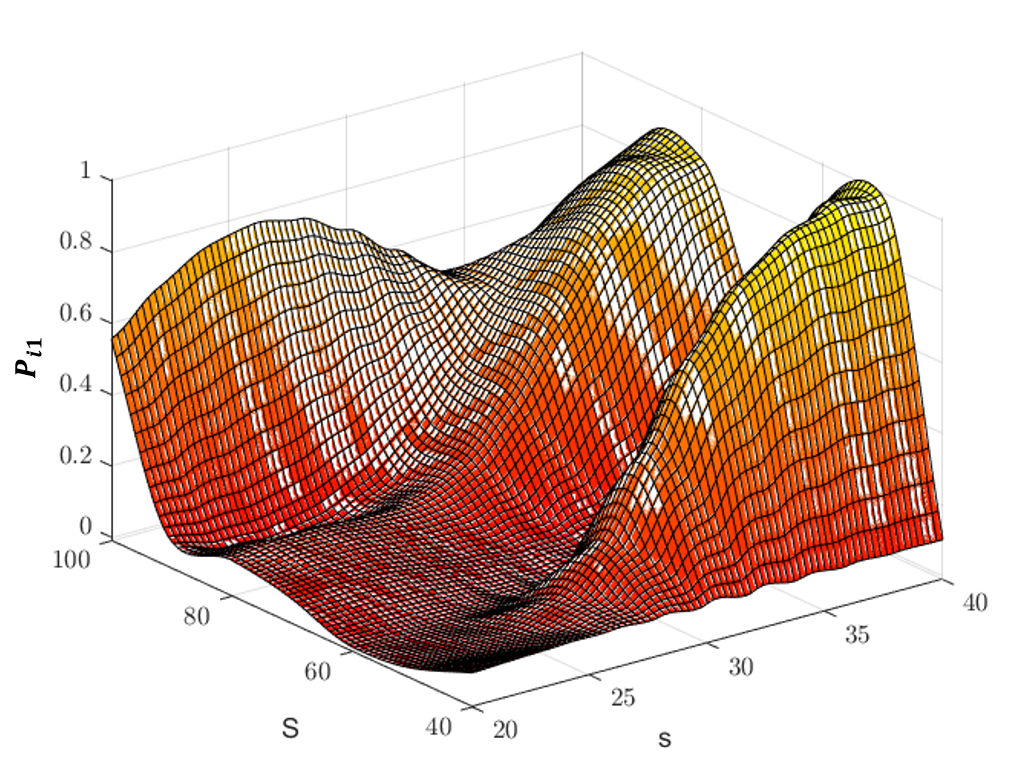}
  \captionsetup{justification=centering}
  \caption{Probability of total yearly cost \\ exceeding 1400.}
  \label{fig:sub1}
\end{subfigure}%
\hspace{0.04\textwidth}
\begin{subfigure}{.45\textwidth}
  \centering
  \includegraphics[width=\linewidth]{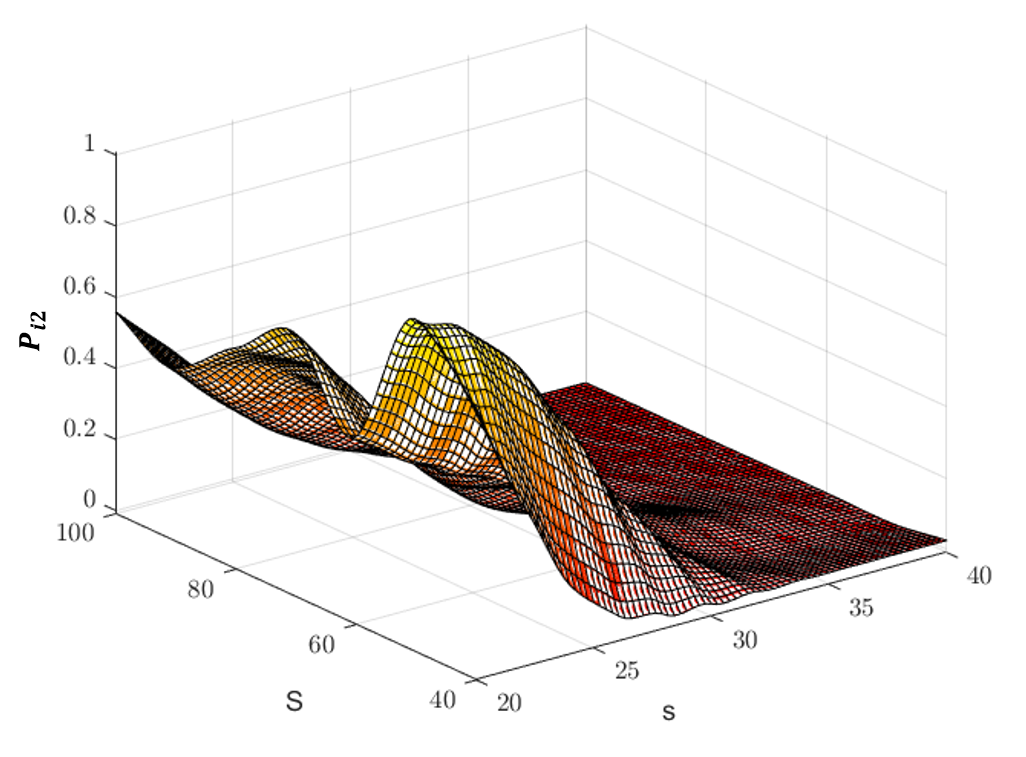}
  \captionsetup{justification=centering}
  \caption{Probability of a stock out occurring \\  within a year.}
  \label{fig:sub2}
\end{subfigure}
\caption{Values of the two performance measures $p_{i1}$ and $p_{i2}$ as functions of $s$ and $S$. }
\label{fig:InventoryConstr}
\end{figure}

\subsection{Additional Results for the Efficiency of ${\cal BRF}$ and ${\cal MPB}_{BN}$ Compared to ${\cal RF}$}
\label{sec:ComparingBeRFandRF}

In this subsection, we provide results for ${\cal BRF}, {\cal MPB}_{BN}$, and ${\cal RF}$ when different values of odds-ratio IZ, batch size, and thresholds are considered. 
Specifically, we conduct experiments where all three procedures consider the same thresholds for each constraint to evaluate their performance under similar conditions. The threshold sets are $h_{\ell} \in \{0.01, 0.05, 0.1, 0.2\}$ for both $\mathcal{BRF}$ and $\mathcal{RF}$ and $\ell=1,2$. We set $\theta_\ell\in \{1.2, 1.5\}$, where $\ell=1,2$, for $\mathcal{BRF}$ and $\mathcal{MPB}_{BN}$, and set $n_0=20$ and $b\in \{32, 100\}$ for $\mathcal{RF}$. These batch sizes strike a balance between achieving high PCD and controlling the total OBS in Table~\ref{tab:brf_rf_sc_nonsc}. The tolerance levels of $\mathcal{RF}$ are set as $\tilde{\epsilon}_\ell$, which is the less conservative of the two settings considered in Section \ref{sec:Thresholds_CompareRF}.
For $\mathcal{MPB}_{BN}$, all four thresholds are tested across $w=2$ passes to examine whether the heuristic approach $\mathcal{MPB}_{BN}$ helps reduce OBS when it needs to assess feasibility with respect to all possible thresholds. More specifically, the first pass uses thresholds $h_{\ell} \in T_\ell^{(1)}= \{0.1, 0.2\}$, while the second pass uses $h_{\ell} \in T_\ell^{(2)}= \{0.01, 0.05\}$ for both $\ell=1,2$. The results for estimated $\widehat{{\rm PCD}}$ and OBS are presented in Table \ref{tab:multiwith2constraints}.

\begin{table}[h!] 

    \centering
    \renewcommand{\arraystretch}{0.8}
    \caption{$\widehat{{\rm PCD}}$ and OBS for the inventory problem.} \label{tab:multiwith2constraints}
    \begin{tabular}{c|clrr}
    \toprule
    & Procedure & & \multicolumn{1}{c}{$\widehat{{\rm PCD}}$} & \multicolumn{1}{c}{OBS (s.e.)} \\ \midrule
    \multirow{4}{*}{$\theta_1 = \theta_2 = 1.2$} & $\mathcal{BRF}$ & & 0.998 & 806,179 (3299.49)  \\ \cline{2-5} 
    & \multirow{2}{*}{$\mathcal{RF}$} & $b=32$ & 1.000 & 6,465,410 (16927.27) \\ \cline{3-5} 
    & & $b=100$ & 1.000 & 6,504,604 (17286.26) \\  \cline{2-5} 
    & $\mathcal{MPB}_{BN}$ & & 0.998 & 716,354 (2236.49) \\ \midrule
    \multirow{4}{*}{$\theta_1 = \theta_2 = 1.5$} & $\mathcal{BRF}$ & & 0.997 & 280,864 (1114.39) \\ \cline{2-5} 
   & \multirow{2}{*}{$\mathcal{RF}$} & $b=32$ & 1.000 & 1,593,432 (3506.43) \\ \cline{3-5} 
   & & $b=100$ & 1.000 & 1,602,894 (3467.99) \\ \cline{2-5}
   & $\mathcal{MPB}_{BN}$ & & 0.997 & 242,411 (741.87) \\  
   \bottomrule
    \end{tabular}
\end{table}

For both choices of $\theta_1=\theta_2$, all three procedures have estimated $\widehat{{\rm PCD}}$ greater than the nominal confidence level 95\%. However, there are significant differences in OBS between the three procedures. With $\theta_\ell = 1.5$, $\mathcal{RF}$ spends approximately about 5.7 times more OBS compared to $\mathcal{BRF}$. 
The differences increase further to about 8 times larger OBS for $\mathcal{RF}$ when $\theta_\ell = 1.2$, demonstrating the advantages of using $\mathcal{BRF}$ or $\mathcal{MPB}_{BN}$ over $\mathcal{RF}$. Note that $\mathcal{RF}$ yields similar OBS for $b=32$ and 100. This is because the OBS is dominated by hard systems and varying batch sizes does not change the required OBS for those systems significantly. Finally, $\mathcal{MPB}_{BN}$ utilizes approximately 12\% fewer OBS compared to $\mathcal{BRF}$. Thus, similar to Table \ref{tab:berfwsingle1}, ${\cal MPB}_{BN}$ is beneficial in this setting, but less so than in Table \ref{tab:multi2const} where ${\cal MPB}_{BN}$ does not need to consider all thresholds. 
Note that all three procedures require smaller OBS when $\theta_\ell=1.5$ than when $\theta_\ell=1.2$, where $\ell=1,2$. This is because a larger $\theta_\ell$ indicates that the decision maker is more indifferent to the feasibility decision near the threshold, which results in earlier stopping of the process. 

Finally, Table \ref{tab:feasible2constraintsBeRF} reports the number of systems deemed feasible by $\mathcal{BRF}$ for one correctly decided macro replication with $\theta = 1.5$ as well as the number obtained from the analytical results derived via 1,000,000 replications of the Markov chain model shown in Figures \ref{fig:TwoPerformance} and \ref{fig:InventoryConstr}. 
\begin{table}[h!]
\centering
\setlength{\tabcolsep}{10pt}
\renewcommand{\arraystretch}{0.8}
{\normalsize
\caption{Number of feasible systems based on $\mathcal{BRF}$'s feasibility decisions when $\theta_{1} = \theta_{2} = 1.5$ as functions of $h_1$ and $h_2$. The analytical results are shown in parentheses.}
\label{tab:feasible2constraintsBeRF}
\begin{tabular}{c|cccc}
\toprule
\backslashbox{$h_1$}{$h_2$} 
& 0.01& 0.05& 0.1& 0.2 \\ \midrule
0.01 & 0 (0) & 0 (0) & 0 (0) & 0 (0) \\ \hline
0.05 & 0 (0) & 2 (2) & 2 (2) & 4 (4) \\ \hline
0.1 & 0 (0) & 5 (3) & 5 (3) & 9 (7) \\ \hline
0.2 & 1 (1) & 11 (11) & 16 (15) & 20 (20) \\ \bottomrule
\end{tabular}
}
\end{table}

Table \ref{tab:feasible2constraintsBeRF} shows that $\mathcal{BRF}$ yields feasibility decisions for each system that are close to the analytical results. 
The slight differences are due to the fact that our approach determines some systems as acceptable. As noted in Section \ref{subsec:CorrectDecision} of our paper, these systems are classified into the $A_\ell(h_{\ell,m}^{(w)})$ set, and even if our feasibility determination for the systems does not exactly match the true decision, it is still regarded as a correct decision. Among the 77 systems, two systems are determined to have the probability $p_{i1}$ of exceeding 1400 less than 5\% and the probability $p_{i2}$ of a stockout occurring within a year less than 5\%. They are the systems with $(s,S) \in \{(30, 70), (32, 70)\}$. The decision maker can utilize this information not only to check feasibility but also to determine the optimal values of $(s,S)$ to minimize the probabilities. For example, if the decision maker introduces tighter thresholds (like 0.03) for which only one of these two systems is feasible, then she can choose the single feasible system as the optimal solution. 

\subsection{Impact of CRN}
\label{subsec:Experiments_CRN}

In this section, we test the impact of CRN when generating observations regarding systems' performance measures. Note that the CRN is not recommended for feasibility determination because it does not involve pairwise comparison. Nevertheless, CRN is often a default design in simulation software packages and the decision maker may want to employ CRN when the feasibility determination is combined with the selection of the best feasible system. In this experiment, we use the same demand for all systems to achieve the impact of CRN. When using CRN, the cross-correlation between systems is estimated by generating 1,000,000 output samples for each system. The results are shown in Figure \ref{fig:crosscorr}. 
\begin{figure}[h!]
\centering
\begin{subfigure}{.5\textwidth}
  \centering
  \includegraphics[width=0.9\linewidth]{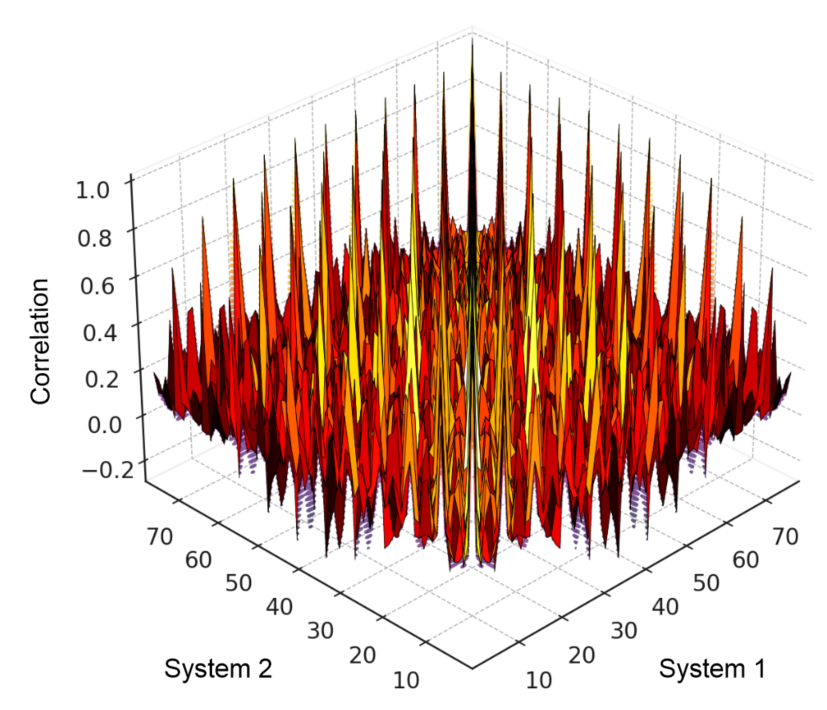}
  \captionsetup{justification=centering}
  \caption{The correlation between systems \\ regarding the first constraint $p_{i1}$.}
  \label{fig:sub1}
\end{subfigure}%
\begin{subfigure}{.5\textwidth}
  \centering
  \includegraphics[width=0.9\linewidth]{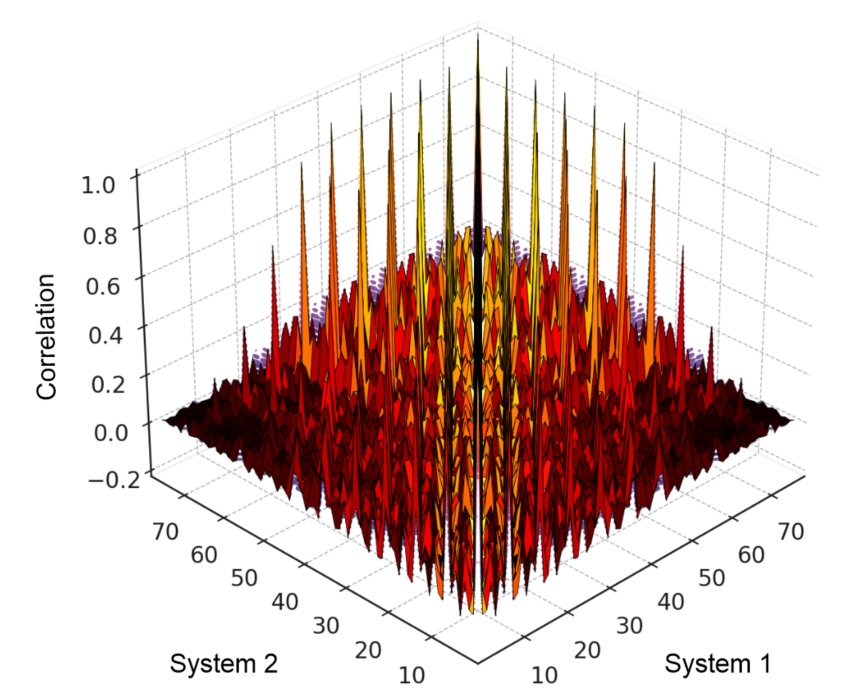}
  \captionsetup{justification=centering}
  \caption{The correlation between systems \\ regarding the second constraint $p_{i2}$.}
  \label{fig:sub2}
\end{subfigure}
\caption{Correlations between systems with respect to the first and the second constraints. }
\label{fig:crosscorr}
\end{figure}

For the first constraint (the probability of total yearly cost exceeding 1400), the mean cross-correlation of the output data across systems was 0.1866, with a maximum value of 1 and a minimum value of -0.2663. For the second constraint (the probability of a stockout occurring within a year), the mean cross-correlation was 0.0781, with a maximum of 1 and a minimum of -0.1042. The higher mean correlation in the first constraint suggests that system outputs exhibit relatively stronger interdependence compared to the second constraint. Additionally, the minimum correlation values indicate that some system pairs exhibit a significant negative correlation for the second constraint. The occurrence of a maximum correlation of 1 can be attributed to the fact that the output of each system follows  Bernoulli distributions; thus, in cases where both systems' outputs are exactly the same with all 0 or all 1, the correlation is computed as 1. 

We now evaluate the performance of our procedures in the context of Appendix \ref{sec:ComparingBeRFandRF} but employing CRN.
In addition, because the proposed procedures require generating dummy Bernoulli random variables (i.e., as $I_{i\ell m n}^{(w)}$ in Algorithms \ref{alg:berf}, \ref{alg:berf_b}, \ref{alg:berf_n}, and \ref{alg:berf_bn}), we also examine a scenario in which the decision maker employs CRN when generating these variables.
Table~\ref{tab:CRNwith2constraints} reports the estimated $\widehat{{\rm PCD}}$ and OBS under three settings: (1) systems are simulated independently (``Without CRN''); (2) systems are simulated using CRN, but the dummy Bernoulli random variables are generated independently (``With CRN''); and (3) systems are simulated using CRN, and the dummy Bernoulli random variables are also generated using CRN (``With CRN, $U_{ir}=U_r$''). Note that we do not test $\mathcal{RF}$ in this setting because $\mathcal{RF}$ performs much worse than $\mathcal{BRF}$ and $\mathcal{MPB}_{BN}$ when systems are simulated independently (see Table \ref{tab:multiwith2constraints}) and  \cite{zhou2022finding} show that incorporating CRN does not reduce OBS as feasibility determination does not involve pairwise comparison and thus there is no variance reduction with CRN. Moreover, the parameter $\eta_\ell$ of $\mathcal{RF}$ increases as $\beta_\ell$ decreases. CRN results in a smaller $\beta_\ell$, a hence larger $\eta_\ell$, and a larger continuation region, which contributes to an increased OBS. 

\begin{table}[h!] 
\centering
\renewcommand{\arraystretch}{0.9}
{\normalsize
\caption{$\widehat{{\rm PCD}}$ and OBS of $\mathcal{BRF}$ and $\mathcal{MPB}_{BN}$ for the inventory problem with CRN.}
\label{tab:CRNwith2constraints}
\begin{tabular}{llrrrr}
\hline
\multicolumn{2}{c}{\multirow{2}{*}{}} & \multicolumn{2}{c}{$\theta_1 = \theta_2 = 1.2$}                   & \multicolumn{2}{c}{$\theta_1 = \theta_2 = 1.5$}                   \\ \cline{3-6} 
\multicolumn{2}{c}{}                  & \multicolumn{1}{c}{$\widehat{{\rm PCD}}$} & \multicolumn{1}{c}{OBS (s.e.)} & \multicolumn{1}{c}{$\widehat{{\rm PCD}}$} & \multicolumn{1}{c}{OBS (s.e.)} \\ \hline
\multirow{3}{*}{$\mathcal{BRF}$} & Without CRN  & 0.998 & 806,178 (3299.49) & 0.997 & 280,864 (1114.39)  \\ 
                      & With CRN   & 0.998 & 805,642 (3313.39) & 1.000 & 281,285 (1104.58)  \\ 
                      & With CRN, $U_{ir}=U_r$ & 1.000 & 799,693 (4240.00) & 1.000 & 281,667 (1448.19) \\ \hline
\multirow{3}{*}{$\mathcal{MPB}_{BN}$}  & Without CRN     & 0.998  &  716,354 (2236.49)  & 0.997  & 242,411 (741.87)   \\ 
                      & With CRN  & 0.999 & 713,420 (2460.22) & 1.000  & 243,045 (768.16)  \\ 
                      & With CRN, $U_{ir}=U_r$ & 1.000 & 714,069 (3262.99) & 1.000 & 242,290 (1020.93)  \\ \hline
\end{tabular}
}
\end{table}

Table \ref{tab:CRNwith2constraints} shows that no significant differences are observed with and without CRN.
Similar to $\mathcal{RF}$, $\mathcal{BRF}$ and $\mathcal{MPB}_{BN}$ also employ a smaller $\beta_\ell$ with CRN but this does not always lead to a larger continuation region because $H$ is an integer that satisfies Equation \eqref{eqn:H}. Indeed, in our experiments, the value of $H$ remains the same with and without CRN. 

\subsection{Additional Results When Decision Maker is Sensitive to Stockout}
\label{sec:SensitiveStockout_Results}

As discussed in Section \ref{sec:Intro}, in this section we assume the decision maker is more sensitive to the stockout probability constraint ($p_{i2}$) than the yearly cost probability constraint ($p_{i1}$) and chooses $h_1 \in \{0.01, 0.05, 0.1, 0.2\}$ and $h_2\in\{0.01, 0.03, 0.05, \ldots, 0.49, 0.51\}$ for $\mathcal{BRF}$ and $\mathcal{RF}$. For $\mathcal{MPB}_{BN}$,  the first pass uses thresholds $h_{1} \in T_1^{(1)} = \{0.1\}$ and $h_{2} \in T_2^{(1)} = \{0.11, 0.21, 0.31, 0.41, 0.51\}$, while the second pass uses $h_{1} \in T_1^{(2)}= \{0.01, 0.05\}$ and $h_{2} \in T_2^{(2)}=\{0.01, 0.03, 0.05, 0.07, 0.09\}$ (since multiple systems pass the feasibility check for thresholds $h_1=0.1$ and $h_2=0.11$ in the first pass). The other parameters are the same as in Table \ref{tab:multi2const}. The results are presented in Table~\ref{tab:multi2constsensitive} below.
\begin{table}[h!] 
\centering
\renewcommand{\arraystretch}{0.8}
\caption{$\widehat{{\rm PCD}}$ and OBS for the inventory problem leveraging $\mathcal{MPB}$ for pruning inferior systems and with more sensitivity to the second constraint.}
\label{tab:multi2constsensitive}
{\normalsize
\begin{tabular}{c|ccc}
\toprule
    & $\mathcal{MPB}_{BN}$ & $\mathcal{BRF}$ & $\mathcal{RF}$      \\ \midrule
$\widehat{{\rm PCD}}$ &  0.998  &  0.997  & 0.999 \\ \hline
${\rm OBS}^{(1)}$ & 77,389 & -- & -- \\ \hline
${\rm OBS}^{(2)}$ & 21,825 & -- & -- \\ \hline
OBS (s.e.) & 99,214 (365.23) & 328,394 (1149.87) & 3,312,674 (7036.43) \\ \bottomrule
\end{tabular}
}
\end{table}

In this scenario, $\widehat{{\rm PCD}}$ consistently meets the desired confidence level of 0.95. Regarding OBS, the results are similar to those in Table \ref{tab:multi2const}: $\mathcal{RF}$ requires 10 times the OBS of $\mathcal{BRF}$ and 33 times that of $\mathcal{MPB}_{BN}$, while $\mathcal{MPB}_{BN}$ successfully performs feasibility determination using 30\% of the OBS required by $\mathcal{BRF}$. Although this scenario is relatively easier than the one in Table \ref{tab:multi2const} due to fewer thresholds, the results confirm that both $\mathcal{MPB}_{BN}$ and $\mathcal{BRF}$ remain efficient procedures. This holds true even when one constraint is considered relatively more critical and is handled with a denser threshold grid than the other constraint. 

\end{document}